\documentclass[a4paper,12pt]{article}

\usepackage{mathtools}
\usepackage[latin1]{inputenc}
\usepackage[all]{xy}
\usepackage{float}
\usepackage{amssymb,amsfonts,amsthm}
\usepackage{enumerate}
\usepackage{epstopdf}
\usepackage{color,fancybox}
\usepackage{url}
\usepackage{graphicx,psfrag,epsfig}
\newcommand{\tM}{\widetilde{\mathbb{M}}}
\newcommand{\tcalM}{\widetilde{\mathcal{M}}}
\newcommand{\m}{m}

\newcommand{\bigo}{{\rm O}}
\newcommand{\smallo}{{\rm o}}

\newcommand{\un}{{\mathbf{1}}}

\newcommand{\V}{\mathbb V}

\renewcommand{\le}{\leqslant}
\renewcommand{\leq}{\leqslant}
\renewcommand{\ge}{\geqslant}
\renewcommand{\geq}{\geqslant}

\newcommand{\cadlag}{c\`adl\`ag }

\newcommand{\A}{\mathbb{A}}

\newcommand{\E}{\mathbb{E}}

\newcommand{\M}{\mathbb{M}}
\newcommand{\R}{\mathbb{R}}

\renewcommand{\P}{\mathbb{P}}
\renewcommand{\L}{\mathbb{L}}
\newcommand{\D}{\mathbb{D}}

\newcommand{\calM}{\mathcal{M}}
\newcommand{\calD}{\mathcal{D}}
\newcommand{\calF}{\mathcal{F}}

\newcommand{\calN}{\mathcal{N}}
\newcommand{\calR}{\mathcal{R}}

\newcommand{\eps}{\varepsilon}
\newcommand{\ph}{\varphi}

\newcommand{\Var}{\mathbb V}
\newcommand{\Card}{{\rm card}}

\renewcommand{\d}{ {\, d \,}}

\newcommand{\eqdef}{:=}
\newcommand{\one}{{\mathbf{1}}}
\newcommand{\dps}{\displaystyle}

\newcommand{\abs}[1]{\mathopen{}\left | #1\right |\mathclose{}}
\newcommand{\set}[1]{\left\{#1\right\}}

\newcommand{\p}[1]{\mathopen{}\left(#1\right)\mathclose{}}
\renewcommand{\b}[1]{\left [ \, #1 \, \right ]}

\newcommand{\norm}[1]{\left\Vert#1\right\Vert}

\theoremstyle{plain}
\newtheorem{The}{Theorem}[section]
\newtheorem{Lem}[The]{Lemma}
\newtheorem{Pro}[The]{Proposition}

\newtheorem{Cor}[The]{Corollary}

\newtheorem{Def}[The]{Definition}

\newtheorem{AssD}{Assumption}

\newtheorem{AssDp}{Assumption}

\newtheorem{AssE}{Assumption}

\numberwithin{equation}{section}

\theoremstyle{definition}
\newtheorem{Rem}[The]{Remark}

\numberwithin{equation}{section}

\begin{document}

\begin{center}
{\sc \Large A Central Limit Theorem for Fleming-Viot Particle Systems with Hard Killing\footnote{This work was partially supported by the French Agence Nationale de la Recherche, under grant ANR-14-CE23-0012, and by the European Research Council under the European Union's Seventh Framework Programme (FP/2007-2013) / ERC Grant Agreement number 614492.}}
\vspace{0.5cm}

\end{center}

{\bf Fr\'ed\'eric C\'erou\footnote{Corresponding author.}}\\
{\it INRIA Rennes \& IRMAR, France }\\
\textsf{frederic.cerou@inria.fr}
\bigskip

{\bf Bernard Delyon}\\
{\it Universit\'e Rennes 1 \& IRMAR, France }\\
\textsf{bernard.delyon@univ-rennes1.fr}
\bigskip

{\bf Arnaud Guyader}\\
{\it Universit\'e Pierre et Marie Curie \& CERMICS, France }\\
\textsf{arnaud.guyader@upmc.fr}
\bigskip

{\bf Mathias Rousset}\\
{\it INRIA Rennes \& IRMAR, France }\\
\textsf{mathias.rousset@inria.fr}
\bigskip

\medskip

\begin{abstract}
\noindent {\rm Fleming-Viot type particle systems represent a classical way to approximate the distribution of a Markov process with killing, given that it is still alive at a final deterministic time. In this context, each particle evolves independently according to the law of the underlying Markov process until its killing, and then branches instantaneously on another randomly chosen particle. While the consistency of this algorithm in the large population limit has been recently studied in several articles, our purpose here is to prove Central Limit Theorems under very general assumptions. For this, we only suppose that the particle system does not explode in finite time, and that the jump and killing times have atomless distributions. In particular, this includes the case of elliptic diffusions with hard killing.
\medskip

\medskip

}

\end{abstract}

\tableofcontents

\section{Introduction}\label{intro}

Let $X=(X_t)_{t\geq 0}$ denote a Markov process evolving in a state space of the form $ F \cup \set{\partial}$, 
where $\partial \notin F$ is an absorbing state. $X$ evolves in $F$ until it reaches $\partial$ and then remains trapped there forever.
Let us also denote $\tau_\partial$ the associated killing time, meaning that
$$\tau_{\partial}\eqdef  \inf\{t\geq 0, X_t=\partial \}.$$
Given a deterministic final time $T>0$, we are interested both in the distribution of $X_T$ given that it has still not been killed at time $T$, denoted
$$\eta_T\eqdef  {\cal L}(X_T|\tau_{\partial }>T),$$
and in the probability of this event, that is
$$p_T\eqdef  \P(\tau_{\partial }>T),$$
with the assumption that $p_T>0$. Without loss of generality, we will assume for simplicity that $\P(X_0 = \partial)=0$ and $p_0 =1$ so that $\eta_0 = {\cal L}(X_0)$. Let us stress that in all this paper, $T$ is held fixed and finite.\medskip 

A crude Monte Carlo method approximating these quantities consists in:
\begin{itemize}
\item simulating $N$ i.i.d.~~random variables, also called particles in the present work, 
$$X_0^1,\dots,X_0^N\ \overset{\rm i.i.d.~}{\sim}\ \eta_0,$$ 
\item letting them evolve independently according to the dynamic of the underlying process $X$, 
\item and eventually considering the estimators
$$\hat{\eta}_T^N\eqdef  \frac{\sum_{i=1}^N \un_{X_T^i\in F}\ \delta_{X_T^i}}{\sum_{i=1}^N\un_{X_T^i\in F}}\hspace{1cm}\mbox{and}\hspace{1cm}\hat{p}_T^N\eqdef  \frac{\sum_{i=1}^N\un_{X_T^i\in F}}{N},$$
with the convention that $0/0= 0$.
\end{itemize}
It is readily seen that these estimators are not relevant for large $T$, typically when $T\gg\E[\tau_{\partial }]$, since one has then to face a rare event estimation problem. A possible way to tackle this issue is to approximate the quantities at stake through a Fleming-Viot type particle system \cite{bhim96,v14}. Under Assumptions~\ref{ass:D} and~\ref{ass:E} that will be detailed below, the following process is well defined for any number of particles $N\geq 2$:
\begin{Def}[Fleming-Viot particle system]\label{def:ips}
The Fleming-Viot particle system $(X^1_t, \cdots, X^N_t)_{t \in [0,T]}$ is the Markov process with state space $F^N$ defined by the following set of rules.
\begin{itemize}
\item Initialization: consider $N$ i.i.d.~~particles
\begin{equation}\label{lkachis}
X_0^1,\dots,X_0^N\ \overset{\rm i.i.d.~}{\sim}\ \eta_0,
\end{equation}
\item Evolution and killing: each particle evolves independently according to the law of the underlying Markov process $X$ until one of them hits  $\partial$ (or the final time $T$ is reached),
\item Branching (or rebirth, or splitting): the killed particle is taken from $\partial $, and is given instantaneously the state of one of the $(N-1)$ other particles (randomly uniformly chosen),
\item and so on until final time $T$.
\end{itemize}
\end{Def}
Finally, we consider the estimators
$$\eta_T^N\eqdef  \frac{1}{N}\sum_{i=1}^N\delta_{X_T^i}\hspace{1cm}\mbox{and}\hspace{1cm}p_T^N\eqdef  \left(1-\tfrac{1}{N}\right)^{N \calN_T},$$
where $N \calN_T$ is the total number of branchings of the particle system until final time $T$. In other words, $\calN_T$ is the empirical mean number of branchings per particle until final time $T$:
$$
\calN_T \eqdef \frac{1}{N} \Card \set{\text{branching times}\leq T}.
$$

Under very general assumptions, Villemonais~\cite{v14} proves among other things that $p_T^N$ (or equivalently $e^{- \calN_T }$) converges in probability to $p_T$ when $N$ goes to infinity, and that $\eta_T^N$ converges in law to $\eta_T$. In~\cite{cdgr1}, we went one step further and established central limit results for $\eta_T^N$ and $p_T^N$. For this, we had to make two specific assumptions. The first one is a ``soft killing'' assumption, meaning that one can define a bounded intensity of being killed when the process is at point $x\in F$. The second one is a so-called ``bounded carr\'e du champ'' assumption and is related to the regularity of the underlying Markov process.\medskip

The purpose of this paper is to generalize the central limit results given in \cite{cdgr1} for $\eta_T^N$ and $p_T^N$ under arguably minimal assumptions. In particular, it includes the case of elliptic diffusive processes killed when hitting the boundary of a given domain. This latter case is usually called ``hard killing'' in the literature and this kind of situation was not covered by \cite{cdgr1}.  \medskip

The rest of the paper is organized as follows. Section 2 details our assumptions, exposes the main results of the paper, and illustrates a possible context of application for a process with hard killing. Section 3 is dedicated to the proof of the central limit theorem, while Section 4 gathers some technical results.

\section{Main result}\label{AMEKCN}
\subsection{Notation and assumptions}\label{zcijo}

For any bounded $\ph:F\to\R$ and $t \in [0,T]$, we consider the unnormalized measure
$$\gamma_t(\ph)\eqdef  p_t \eta_t(\ph)=\E[\ph(X_t)\un_{t<\tau_\partial}],$$
with $X_0\sim\eta_0=\gamma_0$. Note that for any $t\in[0,T]$, one has $p_t=\P(\tau_\partial >t)=\gamma_t(\un_F),$ and recall that $p_0=1$ by assumption. The associated empirical approximation is then given by
$$\gamma_T^N\eqdef  p_T^N \eta_T^N.$$
Remark that $\gamma_0^N=\eta_0^N$. \medskip

For simplicity, we assume that $F$ is a measurable subset of some reference Polish space, and that for each initial condition, $X$ is a \cadlag process in $F \cup \set{\partial}$ satisfying the time-homogeneous Markov property, with $\partial$ being an absorbing state. Its probability transition is denoted $Q$, meaning that there is a semi-group operator $(Q^t)_{t \geq 0}$ defined for any bounded measurable function $\ph:F\to\R$, any $x\in F$ and any $t\geq 0$, by
$$Q^t\ph(x) \eqdef \E[\ph(X_t)|X_0=x].$$
By convention, in the above, the test function $\ph$ defined on $F$ is extended on $F\cup\{\partial\}$ by setting $\ph(\partial)=0$. Thus, we have $Q^t\ph( \partial )=0$  for all $t\geq 0$. This equivalently defines a sub-Markovian semi-group on $F$ also denoted $(Q^t)_{t \geq 0}$. \medskip

Furthermore, for any probability distribution $\mu$ on $F$ and any bounded measurable function $\ph:F\to\R$, the standard notation $\Var_\mu(\ph)$ stands for the variance of the random variable $\ph(Y)$ when $Y$ is distributed according to $\mu$, i.e.
$$\Var_\mu(\ph)\eqdef  \Var(\ph(Y))=\E[\ph(Y)^2]-\E[\ph(Y)]^2=\mu(\ph^2)-\mu(\ph)^2.$$

Our fundamental assumptions can now be detailed. The first one is designed to ensure that two different particles
never jump nor branch at the same time.
\begin{AssD}\label{ass:D} This assumption has two parts:
\begin{enumerate}
\item[(i)] For any initial condition $x \in F$, the jump times of the \cadlag Markov process $t \mapsto X_t\in F \cup \set{\partial}$ have an atomless distribution:
$$
\P\p{ X_{t^-} \neq X_t |X_0=x} = 0 \qquad \forall t \geq 0.
$$
\item[(ii)] There exists a space $\calD$ of bounded measurable real-valued functions on $F$,
which contains at least the indicator function $\un_F$, and such that for any $\ph \in \calD$, the mapping $(x,t) \mapsto Q^{t}(\ph)(x)$ is continuous on $F \times \R_+$.
\end{enumerate}
\end{AssD}

\begin{Rem}\label{alzichachi}
Note that Conditions (i) and (ii) in Assumption~\ref{ass:D} both imply that for any initial condition $x \in F$, the killing time $\tau_\partial$ has an atomless distribution in $[0,+\infty)$. Indeed, for (i), if $t=\tau_\partial$ then obviously $X_{t^-} \neq X_t$ and we conclude that this event happens with probability 0 at any deterministic time $t$. Equivalently, taking $\ph = \un_F$ in~$(ii)$ implies that $t \mapsto \P\p{ \tau_\partial > t |X_0=x}$ is continuous. Note that $ \tau_\partial = + \infty$ may have positive probability.
\end{Rem}

\begin{Rem}
In Section \ref{sec:wellposed}, we present a weaker but less practical version of Assumption~\ref{ass:D}, named Assumption~\ref{ass:Dp}. Lemma \ref{mzoecj} ensures that \ref{ass:D} implies \ref{ass:Dp}. As will be explained, all the results of the present paper are in fact obtained under Assumption~\ref{ass:Dp}.
\end{Rem}

Our second assumption ensures the  existence of the particle system at all time.
\begin{AssE}\label{ass:E} The particle system of Definition~\ref{def:ips} is well-defined in the sense that $\P(\calN_T < + \infty)=1$.
\end{AssE}

The following elementary result will be useful.
\begin{Lem}\label{lem:pcont}
Under Condition~$(i)$ or~$(ii)$ of Assumption~\ref{ass:D}, the non-increasing mapping $t \mapsto p_t=\P\p{\tau_\partial > t}$ 
is continuous and strictly positive on $[0,T]$. Under Assumption \ref{ass:E}, the non-increasing jump process  $t \mapsto p_t^N$ is strictly positive on $[0,T]$.
\end{Lem} 
\begin{proof}
As mentioned in Remark \ref{alzichachi}, Assumptions~\ref{ass:D}$(i)$ and \ref{ass:D}$(ii)$ both ensure the continuity of 
$t \mapsto \P\p{\tau_\partial > t|X_0=x} $ for all $x\in F$.  
And the continuity result now comes from $p_t=\P(\tau_\partial >t)=\int \P(\tau_\partial>t| X_0=x) \eta_0(dx)$. The proof under Assumption~\ref{ass:Dp}(i) is similar. Besides, recall that $p_T$ is strictly positive by assumption. The subsequent assertions are clearly satisfied by definition of $p_t^N$.
\end{proof}

\begin{Rem}
 We will see in Lemma~\ref{lem:decomp} that under Assumptions~\ref{ass:D} and~\ref{ass:E}, one has $\E[p^N_T] = p_T$, which implies in particular that $p_T$ is indeed strictly positive.
\end{Rem}

\subsection{Main result}\label{sec:main}

We keep the notation of Section~\ref{intro}. In particular, $(X^1_t, \ldots, X^N_t)_{t\geq 0}$ denotes the Fleming-Viot particle system. %

\begin{Def}
For any $n \in \set{1, \ldots , N}$ and any $k \geq 0$, we denote by $\tau_{n,k}$
the $k$-th branching time of particle $n$, with the convention $\tau_{n,0} =0$. Moreover, for any $j \geq 0$, we denote by $\tau_{j}$ the $j$-th branching time of the whole system of particles, with the convention $\tau_{0} =0$.
\end{Def}

 Accordingly, the processes
$
\calN^{n}_t\eqdef   \sum_{k\geq 1} \un_{\tau_{n,k}\leq t}
$
and  
$$
\calN_t\eqdef   \frac1N \sum_{n=1}^N \calN^{n}_t = \frac1N \sum_{j\geq 1} \un_{\tau_{j}\leq t}
$$
are \cadlag counting processes that correspond respectively to the number of branchings of particle $n$ before time $t$, and to the total number of branchings per particle of the whole particle system before time $t$.\medskip

As mentioned before, we can then define the empirical measure associated to the particle system as 
$ \eta^N_t\eqdef   \frac{1}{N} \sum_{n=1}^N \delta_{X_t^n},$
while the estimate of the probability that the process is still not killed at time $t$ is denoted 
$p^N_t\eqdef   (1-\tfrac{1}{N})^{N \calN_t},$
and the unnormalized empirical measure is defined as $\gamma^N_t\eqdef   p^{N}_t \eta^N_t$.\medskip

As will be recalled in Proposition \ref{pro:estimate} and already noticed by Villemonais in \cite{v14}, their large $N$ limits are respectively 
$$\eta_t(\varphi)\eqdef\E [\varphi(X_t)|X_t\neq\partial],\;
p_t\eqdef \P(X_t\neq\partial), \mbox{ and }
\gamma_t(\varphi) \eqdef \E[\varphi(X_t)\un_{X_t\neq\partial}].$$
We clearly have $\eta_t(\varphi)=\gamma_t(\varphi)/\gamma_t(\un_F)=\gamma_t(\varphi)/p_t$ and $\gamma_t(\varphi)=\eta_0(Q^t \varphi)$.\medskip

We can now expose the main result of the present paper. As usual, $\calN(m,\sigma^2)$ denotes the normal distribution with mean $m$ and variance $\sigma^2$. As mentioned before, $\Var_{\eta}(\ph)$ denotes the variance of $\ph$ with respect to the distribution $\eta$.

\begin{The}\label{gamma}
Let us denote by $\overline{\cal D}$ the closure with respect to the norm $\norm{\cdot}_{\infty}$ of the space ${\cal D}$ satisfying Condition~$(ii)$ of Assumption~\ref{ass:D}. Then, under Assumptions~\ref{ass:D} and~\ref{ass:E}, for any $\ph$ in $\overline{\cal D}$, one has the convergence in distribution
$$\sqrt{N}\left(\gamma_T^N(\ph)-\gamma_T(\ph)\right)\xrightarrow[N\to\infty]{\cal D}{\cal N}(0,\sigma_T^2(\ph)),$$
where $\sigma_T^2(\ph)$ is defined by
$$\sigma_T^2(\ph) \eqdef p^2_T \Var_{\eta_T}(\ph) - p_T^2\ln(p_T) \, \eta_T(\ph)^2 - 2\int_0^T \Var_{\eta_{t}}(Q^{T-t}(\ph)) p_t dp_t.$$
\end{The}

Since $\un_F\in{\cal D}$ by assumption, and $\gamma_T(\un_F)=p_T$, the CLT for $\eta_T^N$ is then a straightforward application of this result by considering the decomposition 
$$\sqrt{N}\left(\eta_T^N\p{\ph}-\eta_T(\ph)\right)=\frac{1}{\gamma_T^N(\un_F)} \sqrt{N}\left(\gamma_T^N(\ph-\eta_T\left(\ph\right))-\gamma_T(\ph-\eta_T\left(\ph\right))\right),$$
and the fact that $\gamma_T^N(\un_F)$ converges in probability to $p_T=\gamma_T(\un_F)$.

\begin{Cor}\label{eta}
Under Assumptions~\ref{ass:D} and~\ref{ass:E}, for any $\ph$ in $\overline{\cal D}$, one has the convergence in distribution
$$\sqrt{N}\left(\eta_T^N(\ph)-\eta_T\p{\ph}\right)\xrightarrow[N\to\infty]{\cal D}{\cal N}(0,\sigma_T^2(\ph-\eta_T(\ph))/p_T^2).$$
Besides,
$$\sqrt{N}\left(p_T^N-p_T\right)\xrightarrow[N\to\infty]{\cal D}{\cal N}(0,\sigma^2),$$
where 
$$\sigma^2 \eqdef \sigma_T^2(\un_F) = -p_T^2\ln(p_T) - 2\int_0^T \Var_{\eta_{t}}(Q^{T-t}(\un_F)) p_t dp_t.$$
\end{Cor}

\begin{Rem}[Non independent initial conditions]\label{aljcneij}
As will be clear from Step~(i) in the proof of Proposition~\ref{pro:estimate} and from the proof of part (a) of Proposition \ref{lazicj}, Theorem \ref{gamma} and Corollary \ref{eta} both still hold true when the i.i.d.~assumption on the initial condition (\ref{lkachis}) is relaxed and replaced by the following set of conditions: (i) the initial particle system $(X_0^1,\ldots,X_0^N)$ is exchangeable, (ii) its empirical distribution $\eta_0^N=\gamma_0^N$ satisfies 
$$  \E \b{ \left(\eta_0^N(Q^T(\ph))-\eta_0(Q^T(\ph))\right)^2 } \leq c \frac{\norm{\ph}^2_\infty}{N},$$
for some constant $c > 0$, and (iii) the following CLT is satisfied: for any $\ph\in{\cal D}$,
$$\sqrt{N}\left(\eta_0^N(Q^T(\ph))-\eta_0(Q^T(\ph))\right)\xrightarrow[N\to\infty]{\cal D}{\cal N}(0,\V_{\eta_0}(Q^T(\ph))).$$
\end{Rem}

Before proceeding with the proof of Theorem \ref{gamma}, let us give an example of application.

\subsection{Example: Feller process with hard obstacle}

We show in this section how our CLT will apply to Fleming-Viot particle systems based on a Feller process killed when hitting a hard obstacle. As far as we know, this is the first CLT result in that case of ``hard killing''. Yet, there is a cluster of papers studying the hard killing case where $X_t$ is a diffusion process in a bounded domain of $\R^d$ killed when it hits the domain boundary. Among other questions, the convergence of the empirical measures as $N$ goes to infinity is addressed in \cite{BBF12,GK04,lobus} (see also references therein). This case is also included in the general convergence results of \cite{v14}.\medskip 

Let $t \mapsto \widetilde{X}_t$ be a Feller process in a locally compact Polish space $E$, and let $F$ be a bounded open domain with boundary $\partial F = \overline{F} \setminus F$. Let $\tau_\partial$ be the hitting time of 
$E \setminus \overline{F}$, and set $X_t = \tilde{X}_t$ for $t < \tau_\partial$. We consider the set of continuous and bounded functions $\calD = C_b(F)$ extended as usual to $F \cup \set{\partial}$ by setting $\ph(\partial)=0$ if $\ph \in \calD$. Note that $\un_F \in \calD$. \medskip

The difficulty in checking Assumption~\ref{ass:D} is the continuity with respect to $t$ of the mapping $(x,t) \mapsto Q^{t}(\ph)(x)$ because
of the indicator function in
\begin{align*}
Q^t\ph(x)=\E[ \ph(X_t)\un_{t<\tau_\partial}|X_0=x].
\end{align*}
However, we have the following general result:

\begin{Pro}\label{pro:Feller_D}
Assume that $F$ is open, that the process $\widetilde{X}$ is Feller, and the following two conditions: 
\begin{enumerate}
\item[(i)] For all $x \in F$ and all $t \geq 0$, $\P(\widetilde{X}_t \in \partial F|\widetilde X_0=x) = 0$. 
\item[(ii)] For all $x \in \partial F$, $\P( \tau_{\partial} > 0|\widetilde X_0=x) = 0$.
\end{enumerate}
Then Assumption~\ref{ass:D} is fulfilled with $\calD = C_b(F)$.
\end{Pro}
The proof is given in Appendix~\ref{proufefellerd}.
Using the latter, we can prove Assumption~\ref{ass:D} for regular elliptic diffusions.
\begin{Pro}
 Assume that $F$ is open and bounded in $\R^d$ with smooth boundary $\partial F$, and that the diffusion $\widetilde{X}$ 
 has smooth and uniformly elliptic coefficients. Then Assumption~\ref{ass:D} holds true.
\end{Pro}
\begin{proof}
 This is a direct application of Proposition~\ref{pro:Feller_D}. First, the fact that $\widetilde{X}$ is a Feller process can be found for example in \cite{ek86}, Chapter 8, Theorem 1.6.  Next, point (i) is obviously true because the first passage time through $\partial F$ 
of an elliptic diffusion has a density with respect to Lebesgue's measure. Finally, point 
(ii) is also satisfied since the entrance time in the interior of a smooth domain from its boundary by an elliptic diffusion is $0$. This classical fact can for example be proved by applying Itô's formula to a smooth level function defining the domain, and then the law of the iterated logarithm for the Brownian motion.
\end{proof}

Assumption~\ref{ass:E} does not follow from a classical result. It is proved for instance in~\cite{GK12} for regular diffusions and smooth boundary. Note that, in the latter, the authors gives a general set of sufficient assumptions for non explosion, some of them being further generalized in~\cite{v14}. The upcoming result is exactly Theorem~$1$ of Section~$2.1$ in~\cite{GK12}, in the simple case of smooth domains.

\begin{Pro}\label{prop:exp}
Assume that $F$ is open and bounded in $\R^d$ with smooth boundary $\partial F$, and that the diffusion $\widetilde{X}$ has smooth and uniformly elliptic coefficients. Then Assumption~\ref{ass:E} is satisfied.
\end{Pro}

Putting all things together, we conclude that if $F$ is open and bounded in $\R^d$ with smooth boundary $\partial F$, and if the diffusion $\widetilde{X}$ has smooth and uniformly elliptic coefficients, then one can apply the CLT type results of the present paper.

\section{Proof}\label{mazlco}
\subsection{Overview}
The key object of the proof is the c\`adl\`ag martingale 
\begin{equation*}
t \mapsto \gamma^N_t \p{Q} \eqdef \gamma^N_t \p{Q^{T-t}(\ph)},
\end{equation*}
the fixed parameters $T$ and $\ph$ being implicit in order to lighten the notation. Note that, since $\gamma^N_0=\eta^N_0$ and $\gamma_0=\eta_0$,
\begin{align*}
\gamma_T^N(\ph)-\gamma_T(\ph)=\Big(\gamma_T^N(Q)-\gamma^N_0(Q)\Big)
+\Big(\eta^N_0(Q^T(\ph))-\eta_0(Q^T(\ph))\Big)
\end{align*}
is the final value of the latter martingale, with the addition of a second term depending on the initial condition. Note that this second term satisfies a CLT by assumption. We will handle the distribution of $\gamma_T^N(Q)$ in the limit $N\to\infty$ by using 
a Central Limit Theorem for continuous time martingales, namely Theorem~\ref{aeichj}. 
However, this requires several intermediate steps, mainly for the calculation of the quadratic variation $N[\gamma^N\p{Q},\gamma^N\p{Q}]_t$. \medskip

Unfortunately, showing the convergence of this quadratic variation is not easy. Specifically, it is much more difficult than in \cite{cdgr1} where, thanks to the so-called ``carré-du-champ'' and ``soft killing'' assumptions, we could write the predictable quadratic variation as an integral against Lebesgue's measure in time, with bounded integrand. We could then easily show the pointwise convergence of the integrand and apply dominated convergence. 
Here we cannot do that. Instead, the key idea is to replace the quadratic variation by an adapted increasing process $i_t^N$ such that $N[\gamma^N\p{Q},\gamma^N\p{Q}]_t-i_t^N$ is a local martingale. Finally, the convergence of $i_t^N$ requires some appropriate timewise integrations by parts formulas, as well as the uniform convergence in time of $p_t^N$ to $p_t$. 
\medskip

In the sequel, we will make extensive use of stochastic calculus for \cadlag semimartingales, as presented in  \cite{protter} chapter II  or \cite{js03}.

\subsection{Well-posedness and non-simultaneity of jumps}\label{sec:wellposed}
In the remainder, we adopt the standard notation $\Delta X_t=X_t-X_{t^-}$ and, to lighten  the notation, we will denote for $l= 1,2$,
\begin{align}\label{eq:M_t}
 \gamma^N_t \p{Q^l} \eqdef \gamma^N_t \p{ \b{Q^{T-t}(\ph)}^l}.
\end{align}
First, let us fix $T$ and $\ph$, and denote for each $1 \leq n \leq N$ and any $t\in[0,T]$,
\begin{align*}
&\L^n_t \eqdef Q^{T-t}(\ph)(X_{t}^n),%
&\L_t \eqdef \frac{1}{N}\sum_{n=1}^N\L^n_t,
\end{align*}
where, again, the parameters $T$ and $\ph$ are omitted in order to lighten the notation. We start with the following technical assumption, which is the minimal requirement on the non simultaneity of the branchings and jumps times. In particular, Condition~$(i)$ states that a single particle branches at each branching time, making the Fleming-Viot branching rule well-defined.

\begin{AssDp}\label{ass:Dp}  There exists a space $\calD$ of bounded measurable real-valued functions on $F$,
which contains at least the indicator function $\un_F$, and such that for any $\ph\in\calD$, $t \mapsto \L^n_t$ is \cadlag for each $1 \leq n \leq N$, and:
\begin{itemize}
\item[(i)] Only one particle is killed at each branching time:  if $m\neq n$, then $\tau_{m,j}\neq\tau_{n,k}$ almost surely for any $j,k \geq 1$. %
\item[(ii)] The processes $\L_t^m$ and $\L_t^n$ never jump at the same time: if $m\neq n$, then
$$\P( \exists t \geq 0, \, \Delta \L^m_t \Delta \L^n_t \neq 0)= 0.$$
\item[(iii)] The process $\L_t^n$ never jumps at a branching time of another particle: if $m\neq n$, then
$$ \P( \exists  j \geq 0, \, \Delta \L^n_{\tau_{m,j}} \neq 0)= 0.$$
\end{itemize}
\end{AssDp}

As will be shown in Section \ref{lackalicj} in Appendix, it turns out that \ref{ass:D} implies \ref{ass:Dp}. This is stated in the following lemma.

\begin{Lem}\label{mzoecj}
Under Assumption~\ref{ass:D}, 
the system of particles satisfies Assumption~\ref{ass:Dp} with the same set $\calD$ of test functions.
\end{Lem}

Then, under Assumption~\ref{ass:D} or~\ref{ass:Dp}, it is easy to upper-bound the jumps of $\gamma_t^N(Q)$ and $\gamma_t^N(Q^2)$. Indeed, one has
$$|\gamma_t^N(Q)|= |p_t^N\L_t|\leq|\L_t|=\left|\frac{1}{N}\sum_{n=1}^N\L^n_t\right|,$$
and since the jumps of $L_t^n$ and  $L_t^m$ don't coincide, we deduce that
$$|\Delta\gamma_t^N(Q)|\leq\frac{1}{N}\max_{1\leq n\leq N}|\Delta\L^n_t|=\frac{1}{N}\max_{1\leq n\leq N}\left|\Delta Q^{T-t}(\ph)(X_{t}^n)\right|\leq\frac{2\|\ph\|_\infty}{N}.$$ 
The same reasoning applies to $|\Delta\gamma_t^N(Q^2)|$, hence the following result.

\begin{Cor}\label{mzoecjleijf}
Under Assumption~\ref{ass:Dp}, one has $|\Delta\gamma_t^N(Q)|\leq\frac{2\|\ph\|_\infty}{N}$ as well as $|\Delta\gamma_t^N(Q^2)|\leq\frac{\|\ph\|_\infty^2}{N}$.
\end{Cor}  \medskip 

The rest of the paper is mainly devoted the proof of the following result. We recall that $\overline{\cal D}$ is the closure of ${\cal D}$ with respect to the norm $\norm{\cdot}_{\infty}$.

\begin{Pro}\label{techprop}
Under Assumptions~\ref{ass:Dp}  and~\ref{ass:E}, for any $\ph$ in $\overline{\cal D}$, one has
$$\sqrt{N}\left(\gamma_T^N(\ph)-\gamma_T(\ph)\right)\xrightarrow[N\to\infty]{\cal D}{\cal N}(0,\sigma_T^2(\ph)),$$
where
$$\sigma_T^2(\ph) = p^2_T \Var_{\eta_T}(\ph) - p_T^2\ln(p_T) \, \eta_T(\ph)^2 - 2\int_0^T \Var_{\eta_{t}}(Q^{T-t}(\ph)) p_t dp_t.$$
\end{Pro}

Thanks to Lemma~\ref{mzoecj}, the latter yields Theorem~\ref{gamma}.

\subsection{Martingale decomposition}\label{sec:martdec}

This section will build upon the martingale representation of \cite{v14}.

We decompose the process $t \mapsto \gamma^N_t \p{Q}$ into the martingale contributions of the Markovian evolution of particle $n$ between branchings $k$ an $k+1$, which will be denoted $t \mapsto \M^{n,k}_t$, and the martingale contributions of the $k$-th branching of particle $n$, which will be denoted $t \mapsto \calM^{n,k}_t$.\medskip

\begin{Rem}
Throughout the paper, all the local martingales are local with respect to the sequence of stopping times $(\tau_j)_{j\geq 1}$. As required, this sequence of stopping times satisfies $\lim_{j\to\infty}\tau_j>T$ almost surely by Assumption \ref{ass:E}.
\end{Rem}
Recall that we have defined for each $1 \leq n \leq N$ and any $t\in[0,T]$, $\L^n_t \eqdef Q^{T-t}(\ph)(X_{t}^n)$, and $\L_t \eqdef \frac{1}{N}\sum_{n=1}^N\L^n_t$, so that $$\gamma_t^N(Q)= \gamma_t^N(Q^{T-t}(\ph))=p_t^N\L_t.$$

If $\widetilde{X}_t$ is any particle evolving according to the dynamic of the underlying Markov process for (and only for) $t < \tau_\partial$, then it is still true that $Q^{T-t}(\ph)(\widetilde{X}_t)\un_{t < \tau_\partial}$ is a martingale. As a consequence, for any $n\in\{1,\dots,N\}$ and any $k\ge 1$, Doob's optional sampling theorem ensures that by construction of the particle system that the process
\begin{align}\label{lmdkcjm}
\M_t^{n,k} \eqdef \Big(\un_{t<\tau_{n,k}}\L_t^n-\L_{\tau_{n,k-1}}^n\Big)\un_{t\geq\tau_{n,k-1}}=
\begin{cases}
\dps 0 & \text{if } t < \tau_{n,k-1}  \\
\dps {\L^n_{t}} - {\L^n_{\tau_{n,k-1}}} &  \text{if }  \tau_{n,k-1} \leq t < \tau_{n,k} \\
\dps - {\L^n_{\tau_{n,k-1}}} &  \text{if }  \tau_{n,k} \leq t
\end{cases}
\end{align}
is a bounded martingale. Accordingly, under Assumption \ref{ass:E}, the processes
\begin{align}
& \M_t^{n} \eqdef\sum_{k = 1}^\infty \M_t^{n,k}=\L^n_t-\sum_{0 \leq \tau_{n,k}\leq t }\L^n_{\tau_{n,k}},\label{mbbtn}\\
& \M_t \eqdef\frac{1}{\sqrt{N}}\sum_{n=1}^N \M_t^{n},\label{mbbt}
\end{align}
are local martingales.  \medskip

For any $n\in\{1,\dots,N\}$ and any $k\ge 1$, we also consider the process
\begin{equation}\label{aicj}
\calM_t^{n,k} \eqdef
\Big(1-\frac{1}{N}\Big)\Big(\L^n_{\tau_{n,k}}-\frac1{N-1} \sum_{m \neq n}\L^m_{\tau_{n,k}}\Big) 
\un_{t \geq \tau_{n,k}} = \p{ \L^n_{\tau_{n,k}} - \L_{\tau_{n,k}} } \un_{t \geq \tau_{n,k}},
\end{equation}
which by Lemma~\ref{albcios} is a constant martingale with a single jump at $t=\tau_{n,k}$, and which is clearly bounded by $2 \norm{\ph}_\infty$. Then, under Assumption \ref{ass:E}, the processes
\begin{align*}
&\calM_t^{n}\eqdef\sum_{k=1}^\infty\calM_t^{n,k}=\sum_{0 \leq \tau_{n,k}\leq t}\L^n_{\tau_{n,k}} - \L_{\tau_{n,k}},\\
& \calM_t\eqdef\frac{1}{\sqrt{N}}\sum_{n=1}^N\calM_t^n,
\end{align*}
are also local martingales. Recalling the notation $$
\calN^{n}_t\eqdef   \sum_{k\geq 1} \un_{\tau_{n,k}\leq t}
$$
for the number of branchings of particle $n$ before time $t$ and  
$$
\calN_t\eqdef   \frac1N \sum_{n=1}^N \calN^{n}_t = \frac1N \sum_{j\geq 1} \un_{\tau_{j}\leq t}
$$
the total number of branchings per particle before time $t$, \eqref{lmdkcjm} and~\eqref{aicj} respectively implies for each $ 1 \leq n \leq N$ that %
\begin{align}
& d\M^n_t= d \L^n_t-\L_{t}^n d \calN^n_t\label{dmt}\\
&d\calM^n_t =\big(\L^n_{t}- \L_{t} \big) d \calN^n_t,\label{dcalm}
\end{align}
so that in the sum yields
\begin{align}\label{eq:L}
& d\M_t+d \calM_t = \sqrt{N}\p{d \L_t-\L_{t}  \,  d \calN_t}.
\end{align}
Let us emphasize that, in the above equations, $\L_t=\L_{t^+}$ since the process $\L_t$ is right-continuous. \medskip

\begin{Rem}\label{rem:jumps}
 By definition, the jumps of the martingales $\M_t^n$ is included in the union of the set of jumps of $\L_t^n$, and the set of the branching times $\tau_{n,k}$ for $k \geq 1$. The jumps of $\calM_t^n$ are included in the set of the branching times $\tau_{n,k}$ for $k \geq 1$. Therefore, Assumption~\ref{ass:Dp} implies that for $m\neq n$, the jumps of $\M_t^m$ and $\M_t^n$ can't happen at the same time, the same being true for $\M_t^m$ and $\calM_t^n$.
\end{Rem}

The following rule will be useful throughout the paper.
\begin{Lem} Recalling that $p^N_t = (1 - \frac1N)^{N\calN_t}$, it holds that
 \begin{equation}\label{eq:count_diff}
d p^N_t=-p^N_{t^-} d \calN_t .
 \end{equation}
\end{Lem}
\begin{proof}
 One has $\Delta p^N_t = (1 - \frac1N)^{N\calN_t} - (1 - \frac1N)^{N\calN_t-1} = (1 - \frac1N)^{N\calN_t-1} \p{1-\frac1N -1 }$ while $\Delta \calN_t = \frac1N$ for $t = \tau_j$, $j \geq 1$. Hence the result.
\end{proof}

The upcoming result attests that the process $t\mapsto\gamma_t^N(Q)$ is indeed a martingale and details its decomposition.

\begin{Lem}\label{lem:decomp}
We have the decomposition
\begin{equation}\label{eq:decomp}
\gamma_t^N(Q) = \gamma_0^N(Q) +  \frac{1}{\sqrt{N}}\int_0^t p^N_{u^-} \p{ d\M_u + d\calM_u}.
\end{equation}
\end{Lem}
\begin{proof}
Recalling that $p^N_t$ is a piecewise constant process, one has by plain integration by parts
\begin{align*}
\gamma_t^N(Q)= p^N_t \L_t=\gamma_0^N(Q)+ \int_0^t \p{ p^N_{u^-} d\L_u+\L^n_{u} d p^N_u },
\end{align*}
where we emphasize that in the above equation, the last integrand is indeed $\L_u=\L_{u^+}$.
Besides, by~\eqref{eq:count_diff}, we are led to
$$
\gamma_t^N(Q) - \gamma_0^N(Q) = \int_0^t p^N_{u^-} \p{ d \L_u- \L_{u} d \calN_u }.
$$
The result is then a direct consequence of~\eqref{eq:L}.
\end{proof}
\begin{Rem}
Since $\gamma_T(\ph)=\gamma_0( Q^T  \ph)$, this implies the unbiasedness property $\E\b
 {\gamma^N_T(\ph)}=\gamma_T(\ph)$ for all $N \geq 2$. In particular, the case $\ph=\un_F$ gives $\E\b{p^N_T(\ph)}= p_T > 0$.
\end{Rem}

\subsection{Quadratic variation estimates}\label{sec:martquad}

The remarkable fact is that the $2N$ martingales $\set{\M^n_t,\calM^m_t}_{1 \leq n,m \leq N}$ are mutually orthogonal. We recall that two local martingales are orthogonal if their quadratic covariation is again a local martingale.
 
 \begin{Lem} \label{Lem:quad} Under Assumptions~\ref{ass:Dp} and~\ref{ass:E}, the $N^2$ local martingales $\set{\M^n_t,\calM^m_t}_{1 \leq n,m \leq N}$ are mutually orthogonal. In addition,
$$\b{ \calM,\calM}_t = \frac{1}{N}\sum_{n=1}^N\b{ \calM^n,\calM^n}_t.$$
Orthogonality implies that the process $\b{ \M,\cal M}_t$ is a local martingale, and denoting
\begin{align*} 
\A_t \eqdef \frac{1}{N}\sum_{n=1}^N\b{ \M^n,\M^n}_t, 
\end{align*}
that the process $\b{ \M,\M}_t-\A_t$ 
is also a local martingale. In addition, the jumps of $\A$ are controlled by
\begin{align}\label{ajump}
\Delta\A_t \leq \frac{\|\ph\|_\infty^2}{N}.
\end{align}
\end{Lem}
\begin{proof}
By Assumption~\ref{ass:Dp} (see also Remark~\ref{rem:jumps}), for 
$n \neq m$, the piecewise constant martingales $\calM_t^n$ and $\calM_t^m$ do not vary at the same times, so that $\b{\calM^n,\calM^m}_t=0$ and the two martingales are {\it a fortiori} orthogonal. \medskip

In the same manner, for 
$n \neq m$, the martingales $\calM^n_t$ and $ \M^m_t$ do not vary at the same times, so that $\b{\calM^n,\M^m}_t=0$ and the two martingales are {\it a fortiori} orthogonal.\medskip

Moreover, since $\calM^n$ is a pure jump martingale, we have by definition of $\M_t^{n}$
$$d\b{ \M^{n},\calM^{n}}_t =  \Delta\M^n_t d\calM^n_t=- \L_{t^-}^n d\calM^n_t,$$ 
which defines a martingale, so that $\M^n$ and $\calM^n$ are orthogonal.\medskip 

Next, %
we claim that the product $\M^m\M^n$ is a martingale, implying the orthogonality.
Indeed, for a given $s \in [0,T]$, let us define $\sigma_i \eqdef (\tau_i \wedge T) \vee s$ the stopping time in $[s,T]$ closest to the $i$-th branching time. For any $i\ge 1$, conditional to $\mathcal F_{\sigma_{i-1}}$, $(\M^n_t1_{t<\sigma_i})_{t\ge 0}$ 
and $(\M^m_t1_{t<\sigma_i})_{t\ge 0}$ are by construction independent, 
hence
\begin{align*}
\E\left[\left.\M^m_{\sigma_i^-}\M^n_{\sigma_i^-}\right|\mathcal F_{\sigma_{i-1}}\right]=\M^m_{\sigma_{i-1}}\M^n_{\sigma_{i-1}}.
\end{align*}
In addition, since the martingales $\M^m$ and $\M^n$ do not jump simultaneously, it yields $\M^m_{\sigma_i}\M^n_{\sigma_i} = \M^m_{\sigma_i}\M^n_{\sigma^-_i}+\M^m_{\sigma^-_i}\M^n_{\sigma_i} - \M^m_{\sigma^-_i}\M^n_{\sigma^-_i}$, so that
\begin{align*}
 \E[\M^m_{\sigma_i}\M^n_{\sigma_i}|\mathcal F_{\sigma_i^-}]
& =\E[\M^n_{\sigma_i^-}(\M^m_{\sigma_i}-\M^m_{\sigma_i^-})+\M^m_{\sigma_i^-}(\M^n_{\sigma_i}-\M^n_{\sigma_i^-})+\M^m_{\sigma_i^-}\M^n_{\sigma_i^-}|\mathcal F_{\sigma_i^-}]\\
&= \M^m_{\sigma_i^-}\M^n_{\sigma_i^-},
\end{align*}
and combining these equations gives
\begin{align*}
 \E\left[\left.\M^m_{\sigma_i}\M^n_{\sigma_i}\right|\mathcal F_{\sigma_{i-1}}\right]= \M^m_{\sigma_{i-1}}\M^n_{\sigma_{i-1}}.
\end{align*}
By iterating on $i \ge 1$ and taking into account that $\sigma_0 = s$ and $\lim_{i \to +\infty} \sigma_i = T$, we obtain
\begin{align*}
 \E\left[\left.\M^m_T\M^n_T\right|\mathcal F_{s}\right]=  \M^m_{s}\M^n_{s},
\end{align*}
which shows the claimed result. \medskip

For the last point, Assumption~\ref{ass:Dp} guarantees that
$$\Delta\A_t=\frac{1}{N}\max_{1\leq n\leq N}\Delta[\M^n,\M^n]_t=\frac{1}{N}\max_{1\leq n\leq N}\left(\Delta\M_t^n\right)^2,$$
and the indicated result is now a direct consequence of (\ref{lmdkcjm}) and (\ref{mbbtn}).
\end{proof}

In the same way as~\eqref{eq:M_t}, we use in the upcoming lemma the notation for each $t \in [0,T]$,
\begin{align}
 \Var_{\eta^N_{t}}(Q) & \eqdef \Var_{\eta^N_{t}}(Q^{T-t}(\ph))
  = \frac1N \sum_{n=1}^N (\L_{t}^n)^2 - \p{\frac1N \sum_{n=1}^N \L_{t}^n }^{2} . \label{eq:not_var_emp}
\end{align}

\begin{Lem}\label{lemab}One has
  \begin{align}\label{mrondcroch}
  d\b{\calM,\calM}_t\le 4\|\ph\|_\infty^2 d\calN_t.
  \end{align}
  Moreover, there exist a piecewise constant local martingale $\tcalM_t$ and a piecewise constant process $\calR_t$, 
both with jumps at branching times, such that
  \begin{align}%
d\b{\calM,\calM}_t
&=\Var_{\eta^{N}_{t^-}}(Q) d \calN_t + \frac1N d\calR_t + \frac{1}{\sqrt{N}}d\tcalM_t, \label{lemab1}
 \end{align}
 with the following estimate
 \begin{align}\label{lemab2}
|\Delta \calR_{t}|\le  \frac{14 \norm{\ph}_{\infty}^2}{N}.
 \end{align}
\end{Lem}
\begin{proof}
Considering the orthogonality property in Lemma \ref{Lem:quad}, and taking into account that the martingales $\calM^{n,k}$ are piecewise constant with a single jump at time $\tau_{n,k}$, we have
$$\b{\calM,\calM}_t
=\frac1N\sum_{n=1}^N\sum_{k=1}^{+\infty} \p{\calM^{n,k}_{\tau_{n,k}}}^2\un_{t\geq\tau_{n,k}}.$$
This implies (\ref{mrondcroch}) since $|\calM^{n,k}_{\tau_{n,k}}|= |\L^n_{\tau_{n,k}} - \L_{\tau_{n,k}}|\le 2\|\ph\|_\infty$.
This equation also implies that (\ref{lemab1}) holds true with
\begin{align*}
&\tcalM_t:=\frac{1}{\sqrt N}\sum_{n=1}^N\sum_{k=1}^{+\infty}\p{ \big(\calM^{n,k}_{\tau_{n,k}}\big)^2 
-  \E\left[\big(\calM^{n,k}_{\tau_{n,k}}\big)^2  \big| \calF_{\tau_{n,k}^-} \right] }\un_{t \geq \tau_{n,k}}\\
&\calR_t:= \sum_{n=1}^N\sum_{k=1}^{+\infty}\left(\E\left[\big(\calM^{n,k}_{\tau_{n,k}}\big)^2 \big| \calF_{\tau_{n,k}^-} \right] -
\Var_{\eta^N_{\tau_{n,k}^-}}(Q)\right)
\un_{t\geq\tau_{n,k}}.
 \end{align*}
 
On the one hand, Lemma \ref{albcios} ensures that  $\tcalM_t$ is a c\`adl\`ag local martingale. \medskip

On the other hand, by Assumption~\ref{ass:Dp}, we have $\L^l_{\tau_{n,k}}=\L^l_{\tau_{n,k}^-}$ for all $l\neq n$, so that~\eqref{aicj} becomes $\calM^{n,k}_{\tau_{n,k}} = p{1-\tfrac{1}{N}}\left(\L^n_{\tau_{n,k}}- \tfrac{1}{N-1} \sum_{l \neq n}  \L^l_{\tau_{n,k}^-}\right)$. Then, by construction of the branching rule, given $\calF_{\tau_{n,k}^-}$, 
$\L^n_{\tau_{n,k}}$ is uniformly drawn among the $(\L^m_{\tau_{n,k}})_{m\neq n}$, which yields
\begin{align*}
\E\left[\big(\calM^{n,k}_{\tau_{n,k}}\big)^2 \big| \calF_{\tau_{n,k}^-} \right]
&=\tfrac1{N-1}\sum_{m\ne n}\p{1-\tfrac{1}{N}}^2\left(\L^m_{\tau_{n,k}^-}- \tfrac{1}{N-1} \sum_{l \neq n}  \L^l_{\tau_{n,k}^-}\right)^2.
\end{align*}
If we temporarily denote the empirical distribution without particle $n$ by
\begin{align*}
\eta^{(n)}_{t} \eqdef \frac{1}{N-1}\sum_{m\neq n}\delta_{X^m_{t}},
\end{align*}
 we can now reformulate the latter using notation~\eqref{eq:not_var_emp} as
\begin{align*}
\E\left[\big(\calM^{n,k}_{\tau_{n,k}}\big)^2 \big| \calF_{\tau_{n,k}^-} \right]
&=\p{1-\tfrac{1}{N}}^2\ \Var_{\eta^{(n)}_{\tau^{-}_{n,k}}} \p{Q}.
\end{align*}
In other words, we have
$$\calR_t=\sum_{n=1}^N\sum_{k=1}^{+\infty}\left(\p{1-\tfrac{1}{N}}^2\ \Var_{\eta^{(n)}_{\tau^{-}_{n,k}}} \p{Q}-\Var_{\eta^N_{\tau_{n,k}^-}}(Q)\right)\un_{t\geq\tau_{n,k}}.$$

For the last statement, notice that for two probability measures $\mu$ and $\nu$ 
with total variation distance $\|\mu-\nu\|_{tv}$ and for any test function $f$,
\begin{align*}
 \abs{ \Var_{\mu} (f) - \Var_{\nu}(f)}
 \le|(\mu-\nu)(f^2)|+|(\mu-\nu)(f)(\mu+\nu)(f)|
\le 6\|\mu-\nu\|_{tv}\|f\|_\infty^2
\end{align*}
so that, for any $n$ and $k$,
\begin{align*}
 \abs{ \Delta \calR_{\tau_{n,k}}} 
&\le \p{1-\tfrac{1}{N}}^2\Big| \Var_{\eta^{(n)}_{\tau^{-}_{n,k}}} \p{Q} - \Var_{\eta^N_{\tau_{n,k}^-}}(Q) \Big| 
+\Big(1-\p{1-\tfrac{1}{N}}^2\Big)\Var_{\eta^N_{\tau_{n,k}^-}} \p{Q}\\
&\le 6\p{1-\tfrac{1}{N}}^2 \Big((N-1)(\tfrac1{N-1}-\tfrac1N)+\tfrac1N\Big)\|\ph\|_\infty^2
+\frac2N\|\ph\|_\infty^2\\
& \leq \frac{14 \norm{\ph}_{\infty}^2}{N}.\qedhere
\end{align*}
\end{proof}

\begin{Rem} A byproduct of the previous proof is the following equation, which will be useful in Definition~\ref{def:i} below.
\begin{equation}\label{laekcn}
\frac1N \calR_t+ \int_0^t\Var_{\eta^{N}_{s^-}}(Q) d\calN_s=\p{1-\tfrac{1}{N}}^2\frac1N \sum_{n=1}^N\int_0^t   \Var_{\eta^{(n)}_{s^-}}(Q) d\calN^n_s.
\end{equation} 
\end{Rem}

The next lemma is a very important step of the analysis. It relates the quadratic variation of the local martingale $t \mapsto \M_t$ - given, up to a martingale additive term, by the increasing process $t \mapsto \A_t$ defined in Lemma~\ref{lem:decomp} -, with the process $t \mapsto \gamma^N_t(Q^2)$. This will yield estimates on $\A_t$. Note that this idea is inspired by the fact that by definition of the quadratic variation, and for any Markov $X$, the process $t \mapsto \b{Q^{T-t}(\ph)(X_t)}^2$ equals the quadratic variation of the martingale $t \mapsto Q^{T-t}(\ph)(X_t)$ up to a martingale additive term.

\begin{Lem}\label{lemgQ2} 
  There exists a local martingale $(\tM_t)_{t \geq 0}$ such that
  \begin{align}\label{lemgQ21}
  d\gamma^N_t(Q^2) =  p_{t^-}^N d\A_t+ \frac{1}{\sqrt{N}} p_{t^-}^N  d\tM_t.
  \end{align}
  In particular, this implies that
  \begin{align}\label{major}
  \E\left[\int_0^t p_{s^-}^N d \A_s\right]= \E \b{ \gamma^N_t(Q^2) - \gamma^N_0(Q^2)} \leq \norm{\ph}_\infty^2.
  \end{align}
    Moreover, we have
  \begin{align}%
  \E\left[\int_0^t p_{u^-}^Nd[\tM,\tM]_u\right]\le  5\|\ph\|_\infty^4,  \label{lemgQ24}
  \end{align}
  as well as
  \begin{align}
  |\Delta\tM_u|\le  \frac{5\|\ph\|_\infty^2}{\sqrt{N}}.\label{lemgQ25}
  \end{align}
\end{Lem}

\begin{proof}
  Differentiating $\gamma^N_t(Q^2) \eqdef p_t^N \frac1N \sum_{n=1}^N (\L_t^n)^2$ yields
  $$
  d \gamma^N_t(Q^2) = \frac1N \sum_{n=1}^N  p^N_{t^-} d \p{ (\L_t^n)^2 }+(\L_{t}^n)^2 d p^N_t.
  $$
 Since $dp_t^N=-p_{t^-}^N d \calN_t$, one gets
  \begin{align}\label{eq:step_1}
  d \gamma^N_t(Q^2)
  &=\frac1N\sum_{n=1}^N  p^N_{t^-} \left(d \p{ (\L_t^n)^2 }-(\L_{t}^n)^2 d\calN_t\right).
    \end{align}
    
  Next we claim that
  \begin{equation}\label{eq:step_2}
   d(\L_t^n)^2 -  (\L_{t}^n\big)^2 d\calN^n_t = d[\M^n,\M^n]_t + 2\L^n_{t^-}d\M_t^n .
  \end{equation}
  First, know from~\eqref{dmt} that $d\M_t^n=d\L_t^n -\L_{t}^n d \calN^n_t$, so that we can calculate by bilinearity of the quadratic variation
  \begin{align*}
  d[\M^n,\M^n]_t
  &=d[\L^n,\L^n]_t+\big(\L_{t}^n\big)^2d\calN^n_t-2 d \Big [ \int \L^n d \calN^n,\L^n \Big ]_t\\
  &=d[\L^n,\L^n]_t+\big(\L_{t}^n\big)^2d\calN^n_t-2(\Delta \L^n_t)\, \L_{t}^n d\calN^n_t\\
  &=d[\L^n,\L^n]_t+\L_{t}^n\big(2\L_{t^-}^n-\L_{t}^n\big)d\calN^n_t.
  \end{align*}
  Then, using again~\eqref{dmt} through $ d\L_t^n = d\M_t^n +\L_{t}^n d \calN^n_t$, it yields
  \begin{align*}
  d(\L_t^n)^2
  &=2\L^n_{t^-}d\L_t^n+d[\L^n,\L^n]_t\\
  &=\Big(2\L^n_{t^-}d\M_t^n +2\L^n_{t^-}\L_{t}^nd\calN^n_t\Big)+\Big(d[\M^n,\M^n]_t
  -\L_{t}^n\big(2\L_{t^-}^n-\L_{t}^n\big)d\calN^n_t\Big),%
  \end{align*}
  which immediately simplifies into~\eqref{eq:step_2}. \medskip
  
  Putting~\eqref{eq:step_1} and~\eqref{eq:step_2} together, considering in Lemma \ref{Lem:quad} the definition $  \A \eqdef \frac1N \sum_n [\M^n,\M^n]$, and recalling that $\calN \eqdef \frac1N \sum_n \calN^n$, we obtain 
  \begin{align*}
 d \gamma^N_t(Q^2)&=p_{t^-}^N d\A_t+\frac{p_{t^-}^N}{N}\sum_{n=1}^N \b{ (\L_{t}^n)^2  (d\calN_t^n-d\calN_t)+2\L^n_{t^-}d\M_t^n }\\
  &=p_{t^-}^N d\A_t+\frac{p_{t^-}^N}{N}\sum_{n=1}^N \b{ \Big((\L_{t}^n)^2-\frac1N\sum_{m=1}^N(\L_{t}^m)^2\Big)d\calN_t^n
  +2\L_{t^-}^nd\M_t^n },
  \end{align*}
  and we see that \eqref{lemgQ21} is satisfied with
  \begin{align}\label{beq}
  d\tM_t %
  &= \frac{1}{\sqrt{N}}\sum_{n=1}^N J^n_t d\calN_t^n
  +2\L_{t^-}^nd\M_t^n ,
  \end{align}
where we have defined 
$$J^n_t \eqdef \p{\L_{t}^n}^2-\frac1N\sum_{m=1}^N(\L_{t}^m)^2 = (1-1/N) \p{ \p{\L_{t}^n}^2 -\frac{1}{N-1}\sum_{m \neq n }(\L_{t}^m)^2}.$$
Note that, in the same fashion as $(\calM^n_t)_{t \geq 0}$, $ \p{ \int_0^t J^n_s d \calN_s^n }_{t \geq 0}$ is a local martingale since, for each $k\geq 1$, $ t \mapsto J^n_t \un_{t \geq \tau_{n,k}}$ is a bounded martingale by definition of the branching rule and Lemma~\ref{albcios}. Moreover, in the same way as in Lemma~\ref{lem:decomp}, the $2N$ local martingales 
$$\set{\p{\int_0^t\L_{s^-}^m d \M_s^m }_{t \geq 0},\p{ \int_0^t J^n_s d \calN_s^n }_{t \geq 0} }_{1\leq n, m \leq N}$$ are all orthogonal to each other. Indeed, for any pair $n,m$, (i) $ [\M^n,\M^m]$ is a martingale by Lemma~\ref{lem:decomp}, (ii) $[\calM^n,\M^m]=0$ and $[\calM^n,\calM^m]=0$ if $n \neq m$ by Assumption~\ref{ass:Dp}. The only new point to check is that the quadratic covariation 
$$d \b{\int \L_{s^-}^n d \M_s^n ,\int J^n_s d \calN_s^n  }_{t} = - \p{\L_{t^-}^n}^2 J^n_t d \calN_t^n$$ 
is indeed a local martingale, which is a consequence of the branching rule implying $\E\b{J^n_{\tau_{n,k}} | \calF_{\tau_{n,k}^-}}=0$ and Lemma~\ref{albcios}.\medskip  
  
To establish~\eqref{lemgQ24} and~\eqref{lemgQ25}, we recall that for any $1\leq n\leq N$ $\sup_{t\geq 0}|J_t^n|\leq\|\ph\|_\infty^2$, $\sup_{t\geq0}|\L_{t^-}^n|\leq\|\ph\|_\infty$, and $\sup_{t\geq 0}|\Delta\M_t^n|\leq 2\|\ph\|_\infty$.\medskip

For~\eqref{lemgQ24}, we apply Itô's isometry to~\eqref{beq} and use orthogonality to get
  \begin{align*}
  \E\int_0^t p_{u^-}^Nd[\tM,\tM]_u & = \frac{1}{N}\sum_{n=1}^N\E \b{ \int_0^u p_{u^-}^N \p{J^n_u}^2 d\calN_u^n  + 4 \int_0^t p_{u^-}^N (\L_{t^-}^n)^2d[\M^n,\M^n]_u} \\
  &\le \|\ph\|_\infty^4 \E\left[\int_0^t p_{u^-}^Nd\calN_u\right] + 4\|\ph\|_\infty^2 \E\left[\int_0^t p_{u^-}^Nd\A_u\right]\\
  &\le\|\ph\|_\infty^4\frac1N \sum_{j=1}^\infty\big(1-\tfrac1N\big)^{j-1} +  4\|\ph\|_\infty^2 \E\left[\gamma_t^N(Q^2)-\gamma_0^N(Q^2)\right]\\\
  &\le 5\|\ph\|_\infty^4 .
  \end{align*}
  
  In order to obtain~\eqref{lemgQ25}, consider (\ref{beq}) and recall from Assumption~\ref{ass:Dp} that, for $n\neq m$, $\Delta\calN_t^n\Delta\calN_t^m=0$ and $\Delta\M_t^n\Delta\M_t^m=0$. We then deduce that
 $$|\Delta\tM_t|\le \frac{1}{\sqrt{N}}\left(\sup_{t\geq 0}|J_t^n|+2\sup_{t\geq0}|\L_{t^-}^n|\sup_{t\geq 0}|\Delta\M_t^n|\right)\le \frac{5\|\ph\|_\infty^2}{\sqrt{N}}.$$ 
 \qedhere

 \end{proof}

\subsection{$\L^2$-estimate}

The convergence of $\gamma^N_T(\ph)$  to $\gamma_T(\ph)$ when $N$ goes to infinity is now a direct consequence of the previous results. This kind of estimate was already noticed by Villemonais in~\cite{v14}.

\begin{Pro}\label{pro:estimate}
For any $\ph \in {\calD}$, we have
\begin{align*}
\E \b{ \p{ \gamma^N_T(\ph) - \gamma_T(\ph) }^2 } \leq \frac{6 \norm{\ph}_\infty^2}{N}.
\end{align*}
\end{Pro}

\begin{proof} 
Thanks to Lemma \ref{lem:decomp} and the fact that $\gamma_T(\ph)=\gamma_0( Q^T  \ph)$, we have the orthogonal decomposition
$$\gamma^N_T(\ph) - \gamma_T(\ph)  =\frac{1}{\sqrt{N}} \int_0^T p^N_{t^-}\ d\M_t + \frac{1}{\sqrt{N}} \int_0^T p^N_{t^-}\ d\calM_t +  \gamma^N_0 (Q^T\ph) - \gamma_0( Q^T  \ph),$$
and it is easy to upper-bound the individual contribution of each term to the total variance.

(i)~Initial condition. Since $\gamma_0=\eta_0$ and $\gamma_0^N=\eta_0^N$, we have
$$\E\b{ \p{\gamma^N_0 (Q^T\ph) - \gamma_0 (Q^T  \ph)}^2  }=\tfrac1N\Var_{\eta_0}(Q^T(\ph)(X))
\leq \tfrac1N\|Q^T(\ph)\|_\infty^2\leq \tfrac1N\|\ph\|_\infty^2.$$

(ii)~$\calM$-terms. Using Itô's isometry and (\ref{mrondcroch}), we obtain
\begin{align*}
\E\left[\left(\int_0^T p_{t^-}^N d\calM_t\right)^2\right]
&=\E\left[\int_0^T \big(p_{t^-}^N\big)^2 d[\calM,\calM]_t\right]\\
&\leq4\|\ph\|_\infty^2 \frac1N \sum_{j=1}^{\infty} \p{1-\tfrac{1}{N}}^{2(j-1)}\leq 4\|\ph\|_\infty^2. 
\end{align*}

(iii)~$\M$-terms. In the same way, applying Itô's isometry and (\ref{lemgQ21}), we get
\begin{align*}
\E\left[\left(\int_0^T p_{t^-}^N d\M_t\right)^2\right]
&=\E\left[\int_0^T \big(p_{t^-}^N\big)^2 d[\M,\M]_t\right]\\
&\le\E\left[\int_0^T p_{t^-}^N d\A_t\right]=\E\left[\gamma_T^N(Q^2)\right]\leq \|\ph\|_\infty^2.
\end{align*}
\end{proof}

In particular, Proposition~\ref{pro:estimate} implies that for any $\ph$ in ${\cal D}$, $\gamma^N_t(\ph)$ converges in probability to $\gamma_t(\ph)$ when $N$ goes to infinity. Since we have assumed that $\one_F$ belongs to ${\cal D}$, the probability estimate $p_t^N$ goes to its deterministic target $p_t$ in probability. The next subsection provides a stronger result.

\subsection{Time uniform estimate for $p_t$}

In this section, we prove the convergence of $\sup_{t \in [0,T]} \abs{p^N_t - p_t}$ to $0$ in probability by using the time marginal convergence of Proposition~\ref{pro:estimate}. Recall that, by Assumption~\ref{ass:D} or~\ref{ass:Dp}, the mapping $t\mapsto p_t$ is continuous (Lemma~\ref{lem:pcont}. Hence, the proof only uses this argument and the monotonicity of $t \mapsto p^N_t$. One can merely see it as a Dini-like result.

\begin{Lem}\label{lzich}
One has
$$\sup_{t \in [0,T]} \abs{p^N_t - p_t}\xrightarrow[N\to\infty]{\P}0.$$ 
\end{Lem}

\begin{proof} 
Since the mapping $t\mapsto p_t$ is continuous on $[0,T]$ by Lemma~\ref{lem:pcont}, it is uniformly continuous. Hence, for any $\varepsilon>0$, there exists a subdivision $\{t_0=0<t_1<\dots<t_J=T\}$ such that, for any $1\leq j\leq J$ and any $t$ in $[t_{j-1},t_j]$, one has 
$$\max(|p_t-p_{t_{j-1}}|,|p_t-p_{t_{j}}|)\leq\varepsilon.$$
Hence, since $t\mapsto p_t^N$ is decreasing, it is readily seen that
$$|p_t^N-p_t|\leq\max(\abs{p_{t_{j-1}}^N-p_t},\abs{p_t-p_{t_{j}}^N})\leq\varepsilon+\max(|p_{t_{j-1}}^N-p_{t_{j-1}}|,|p_{t_{j}}^N-p_{t_{j}}|).$$
Consequently, with probability $1$, uniformly in $t\in[0,T]$, we get
$$|p_t^N-p_t|\leq\varepsilon+\max_{0\leq j\leq J}|p_{t_{j}}^N-p_{t_{j}}|.$$
Taking $\ph=\un_F$ and $T=t_j$ in Proposition \ref{pro:estimate} ensures that
$$\max_{0\leq j\leq J}|p_{t_{j}}^N-p_{t_{j}}|\xrightarrow[N\to\infty]{\P}0.$$
Therefore, %
we have $\P(\sup_{t \in [0,T]}|p_t^N-p_t|>2\varepsilon)\leq\P(\max_{0\leq j\leq J}|p_{t_{j}}^N-p_{t_{j}}|>\varepsilon)\rightarrow 0$ when $N \to +\infty$.
Since $\varepsilon$ is arbitrary, we get the desired result.
\end{proof}

\subsection{Approximation of the quadratic variation}

As will become clear later, the following process represents a useful approximation 
of $N\b{\gamma^N(Q),\gamma^N(Q)}_t$. 
\begin{Def}\label{def:i} 
 For each $\ph \in \calD$ and $T > 0$ given, we define for $t \in [0,T]$ the \cadlag increasing process
 \begin{align}\label{indef}
i^N_t \eqdef \int_0^t\p{p^N_{u^-}}^2 d\A_u-\int_0^t\Var_{ \eta^{N}_{u^-}}(Q)p^N_{u^-}dp^N_u+\frac1N  \int_0^t \big(p^N_{u^-}\big)^2 d \calR_u .
\end{align}
\end{Def}
The fact that this process is increasing comes from (\ref{laekcn}) and $dp_t^N=-p_{t^-}^N d \calN_t$, which yields the alternative formulation
$$ - \Var_{ \eta^{N}_{t^-}}(Q)p^N_{t^-}dp^N_t + \frac1N  \big(p^N_{t^-}\big)^2 d \calR_t  =  
 \big(p^N_{t^-}\big)^2 \frac{(1-1/N)^2}{N} \sum_{n=1}^N \Var_{ \eta^{(n)}_{t^-}}(Q) d \calN^n_t $$
 where the empirical distribution without particle $n$ is denoted by
$\eta^{(n)}_{t} \eqdef \frac{1}{N-1}\sum_{m\neq n}\delta_{X^m_{t}}$. \medskip

The estimation of $i^N_t$ is in fact easier than the estimation of 
$N\b{\gamma^N(Q),\gamma^N(Q)}_t$ and these two increasing processes are equal up to a martingale term.
\begin{Lem}\label{lem:i_mart}
The process $ t \mapsto N \b{\gamma^N(Q) , \gamma^N(Q)}_t - i^N_t $
is a local martingale.
\end{Lem}
\begin{proof}
From (\ref{eq:decomp}) and Lemma \ref{Lem:quad}, we know that
\begin{align}
 N \b{\gamma^N(Q) , \gamma^N(Q)}_t - \int_0^t \big(p^N_{u^-}\big)^2  d\A_u  
 - \int_0^t \big(p^N_{u^-}\big)^2  d \b{\calM,\calM}_u  \label{paiozjc}
\end{align}
is a local martingale. The result is then a direct consequence of~\eqref{lemab1}.
\end{proof}

The next step is just a reformulation of $i_t^N$ through an integration by parts.

\begin{Lem}\label{oaichmp}
 The increasing process $i^N_t$ can be decomposed as 
\begin{align*}
 i^N_t =&\ p^N_t \gamma_t^N\p{Q^2} - \gamma_0^N\p{Q^2} + \b{\gamma_t^N\p{Q}}^2\ln p^N_t  -  2 \int_0^t \gamma^N_{u^-}(Q^2) dp^N_u\\
 &+ \m_t^N + \ell^N_t + O\p{\frac{1}{N}},
\end{align*}
where
\begin{align*}
\m_t^N \eqdef -\frac{1}{\sqrt{N}}\int_0^t  \p{p^N_{u^-}}^2 d\tM_u
\end{align*}
is a local martingale, and
\begin{align*}
\ell^N_t \eqdef - \int_0^t \ln p^N_{u^-} d\p{\gamma_{u}^N\p{Q} }^2.
\end{align*}
\end{Lem}
\begin{proof}

Starting from (\ref{indef}), we apply Lemma~\ref{lemgQ2} to get
\begin{align*}
 i^N_t = \int_0^t p^N_{u^-} d\gamma_u^N\p{Q^2} -  \int_0^t \Var_{\eta^N_{u^-}}(Q)p^N_{u^-} dp^N_u + \m_t^N+\frac{1}{N}\int_0^t \big(p^N_{u^-}\big)^2  d \calR_u.
\end{align*}
Using (\ref{lemab2}), we are led to
\begin{align*}
\Big| \int_0^t \big(p^N_{u^-}\big)^2  d \calR_u \Big|
 \le \frac{14\|\ph\|_\infty^2}N\sum_{i=0}^\infty(1-\tfrac1N)^{2i}\le 7\|\ph\|_\infty^2,
\end{align*}

We claim now that a first timewise integration by parts (IBP) yields
\begin{align*}
\int_0^t p^N_{u^-} d\gamma_u^N\p{Q^2} = - \int_0^t \gamma_{u^-}^N\p{Q^2} dp^N_{u} + 
\gamma_t^N\p{Q^2} p^N_t - \gamma_0^N\p{Q^2} + O \p{\frac{1}{N}}.
\end{align*}
Indeed, Assumption~\ref{ass:Dp} implies that $|\Delta \gamma_{\tau_j}^N\p{Q^2}| \leq 2 \norm{\ph}_{\infty}^2/N$ so that Condition~(i) of Lemma \ref{oachz} is satisfied with $z^N_t=\gamma_t^N\p{Q^2}$ and IBP rule~(\ref{eq:ipp1}) can therefore be applied. \medskip

Next, remarking that 
$$\int_0^t \Var_{\eta^N_{u^-}}(Q)p^N_{u^-} dp^N_u=\int_0^t \gamma_{u^-}^N\p{Q^2} dp^N_{u}-\int_0^t \p{ \gamma_{u^-}^N\p{Q} }^2 \p{p^N_{u^-}}^{-1} dp^N_{u},$$
a second timewise IBP yields
\begin{align*} 
 \int_0^t  \p{ \gamma_{u^-}^N\p{Q}}^2 \p{p^N_{u^-}}^{-1} dp^N_{u} &= \int_0^t  \p{ \gamma_{u^-}^N\p{Q}}^2 d \log p^N_{u}  +O(\frac1N) \\
&= \b{\gamma_{t}^N\p{Q}}^2 \ln p^N_t - \int_0^t \ln p^N_{u^-} d\p{\gamma_{u}^N\p{Q} }^2 + O(\frac1N).
\end{align*}
Indeed, %
Assumption~\ref{ass:Dp} also implies, for $N\geq 2$, 
$$|\gamma^N_{\tau_j^-}(Q)| \leq 2\norm{\ph}_{\infty}(1-1/N)^{j}\hspace{.5cm}\mbox{and}\hspace{.5cm}|\Delta \gamma^N_{\tau_j}(Q)| \leq \frac{6\norm{\ph}_{\infty}}{N}(1-1/N)^{j-1}.$$ 
As a consequence, Conditions~(ii) and~(iii) of Lemma \ref{oachz} are satisfied for $z^N_t=\p{\gamma^N_t(Q)}^2$, so that we can apply successively rules~(\ref{eq:ipp2}) 
and~(\ref{eq:ipp3}) of  Lemma~\ref{oachz}. Finally, putting all estimates together gives the desired result.
\end{proof}

\begin{Lem}\label{pamich}
 One has $\E\b{\p{\m_t^N}^2} = O\p{1/N}$ as well as $\E \abs{\ell^N_t} = O\p{1/\sqrt{N}}$.
\end{Lem}
\begin{proof}
The first assertion is an immediate consequence of Itô's isometry for martingales, together with~(\ref{lemgQ24}). For the second one, Itô's formula yields
\begin{align*}
\ell^N_t \eqdef -2 \int_0^t \ln p^N_{u^-} \gamma^N_{u^-}(Q)   d\gamma^N_{u}(Q) - \int_0^t \ln p^N_{u^-} d\b{ \gamma^N(Q) , \gamma^N(Q)}_u.
\end{align*}
Therefore,
$$\E\left|\ell^N_t\right|\leq2 \E\left[\left|\int_0^t \ln p^N_{u^-} \gamma^N_{u^-}(Q)   d\gamma^N_{u}(Q)\right|\right] + \E\left[\int_0^t |\ln p^N_{u^-}| d\b{ \gamma^N(Q) , \gamma^N(Q)}_u\right].$$
Then,  Cauchy-Schwarz
 inequality and Itô's isometry provide
\begin{align*}
\E\left|\ell^N_t\right|\leq&\ 2 \left(\E\left[\int_0^t \left(\ln p^N_{u^-} \gamma^N_{u^-}(Q)\right)^2  d\b{ \gamma^N(Q) , \gamma^N(Q)}_u \right]\right)^{1/2}\\
 &+ \E  \left[\int_0^t |\ln p^N_{u^-}| d\b{ \gamma^N(Q) , \gamma^N(Q)}_u\right].
\end{align*}
Since $p^2|\ln p|  \leq 1$ for any $p\in(0,1]$, we have
$$\left(\ln p^N_{u^-} \gamma^N_{u^-}(Q)\right)^2=\left(\ln p^N_{u^-} \times p^N_{u^-}\eta^N_{u^-}(Q)\right)^2\leq |\ln p^N_{u^-}|\times\|\ph\|_\infty^2.$$
Hence, if we denote
\begin{align*}
 c(N) := \E \left[\int_0^T \left|\ln p^N_{u^-}\right| d\b{ \gamma^N(Q) , \gamma^N(Q)}_u\right],
\end{align*}
it comes
$$\E\left|\ell^N_t\right|\leq2\|\ph\|_\infty\sqrt{c(N)}+c(N).$$
Next, the basic decomposition of Lemma~\ref{lem:decomp} yields $d \b{ \gamma^N(Q) , \gamma^N(Q)}_t = \frac1N \p{p^N_{t^-}}^2 \b{\M,\M}_t + \frac1N \p{p^N_{t^-}}^2 \b{\calM,\calM}_t$, so that the orthogonality property~\ref{Lem:quad} allows us to reformulate $c(N)$ as
\begin{align*}
 c(N) =\frac1{N}\E\left[\int_0^t \left|\ln p^N_{u^-}\right| \p{p^N_{u^-}}^2 (d\A_u+d\b{\calM,\calM}_u)\right].
\end{align*}
Using the fact that  $p |\ln p|  \leq 1$ together with~\eqref{mrondcroch}, it yields 
\begin{align*}
 c(N) \leq \frac1{N}\E \b{\int_0^T p^N_{t^-} d \A_t}+\frac4{N^2}\|\ph\|_\infty^2 \sum_{j\geq 1} (1-\tfrac1N)^j,
\end{align*}
so that~\eqref{major} gives $c(N)\le \frac5N\|\ph\|_\infty^2$ and the proof is complete.
\end{proof}

\subsection{The asymptotic variance and the convergence}

For the forthcoming calculations we recall that
\begin{align*}%
p_t \Var_{\eta_t}\p{Q^{T-t}(\ph)} &=   \gamma_t\b{ \p{Q^{T-t}(\ph)}^2} - p_t^{-1}\b{ \gamma_t\p{Q^{T-t}(\ph)}}^2\\
& = \gamma_t\b{ \p{Q^{T-t}(\ph)}^2} - p_t^{-1} \p{\gamma_T(\ph)}^2\\
& = \gamma_t\p{Q^2} - p_t^{-1} \p{\gamma_T(\ph)}^2. 
\end{align*}

The asymptotic variance formula will be denoted as follows:

\begin{Def} For any $t \in[0,T]$ and any $\ph\in{\cal D}$, let us define
\begin{align}
i_t(\ph) \eqdef & p_t\gamma_t(Q^2) - \gamma_0(Q^2) + \b{\gamma_t(Q)}^2 \ln p_t - 2 \int_0^{t}\gamma_u\p{Q^2} dp_u.\label{eq:var_1}
 \end{align}
\end{Def}

Our next purpose is to show that $i_t(\ph)$ corresponds to the asymptotic variance of  interest, as suggested by Lemma~\ref{oaichmp}.

\begin{Pro}\label{almchi}
For any $t \in[0,T]$, one has
\begin{align*}
 i^N_t \xrightarrow[{N \to\infty}]{\P} i_t(\ph).
\end{align*}
\end{Pro}

\begin{proof}
By Lemma \ref{oaichmp} and the relation $\gamma_t\p{Q^{T-t}(\ph)}=\gamma_t\p{Q}=\gamma_T\p{\ph}$, we can write
\begin{align*}
i^N_t-i_t(\ph)=&\Big(p^N_t \gamma_t^N\p{Q^2}-p_t \gamma_t\p{Q^2}\Big) 
-\Big(\gamma_0^N\p{Q^2}-\gamma_0\p{Q^2}\Big)\\
& + \Big(\gamma_t^N\p{Q}^2\ln p^N_t-\gamma_t\p{Q}^2\ln p_t\Big) +  \m_t^N + \ell^N_t+ O\p{1/N}\\
& -  2 \left(\int_0^t \gamma^N_{u^-}(Q^2) dp^N_u -\int_0^t \gamma_{u}(Q^2) dp_u\right) .
\end{align*}
Clearly, by Proposition~\ref{pro:estimate} and Lemma~\ref{pamich}, 
the boundary terms and the rest terms all tend to $0$ in probability. So we just have to show that
$$\int_0^t \gamma^N_{u^-}(Q^2) dp^N_u -\int_0^t \gamma_{u}(Q^2) dp_u=a^N_t+b^N_t$$
goes to $0$ as well, where we have defined 
$$a^N_t:= \int_0^t \gamma_{u^-}^N(Q^2)d(p^N_u-p_u)\qquad\mbox{and}\qquad b^N_t:=\int_0^{t} \p{\gamma^N_{u^-}(Q^2) -\gamma_{u}(Q^2) } dp_u.$$

The convergence of $b_t^N \xrightarrow[]{\P} 0$ is a direct consequence of Proposition~\ref{pro:estimate}. The proof of $a_t^N \xrightarrow[]{\P} 0$ requires more attention. Since $|\Delta \gamma_{\tau_j}^N(Q^2)| \leq 2 \norm{\ph}_\infty/N$, the timewise integration by parts rule~\eqref{eq:ipp1} of Lemma \ref{oachz} enables us to rewrite the first term as 
$$
a_t^N= -\int_0^t (p^N_{u^-}-p_u) \d\gamma_{u}^N(Q^2)+\gamma_{t}^N(Q^2)(p^N_t-p_t)+O\p{1/N},
$$
where we have used $p^N_0=p_0=1$. Since $\gamma_{t}^N(Q^2)$ is bounded, the boundary term goes to $0$ by Proposition \ref{pro:estimate}. For the integral term, equation (\ref{lemgQ21}) leads to the decomposition
\begin{equation}
\int_0^t (p^N_{u^-}-p_u) \d\gamma_{u}^N(Q^2)
= \int_0^t (p^N_{u^-}-p_u)p^N_{u^-} d\A_u+\frac{1}{\sqrt{N}}\int_0^t (p^N_{u^-}-p_u)p^N_{u^-} d\tM_u.\label{liashjec}
\end{equation}
Since $\A$ is an increasing process, it comes
\begin{equation}
\left|\int_0^t (p^N_{u^-}-p_u) p^N_{u^-}d\A_u\right|
\leq \sup_u |p^N_{u^-}-p_u|\times\left(\int_0^t p^N_{u^-}d\A_u\right).\label{lazeich}
\end{equation}
The supremum term goes to $0$ in probability by Lemma \ref{lzich} and, by~\eqref{lemgQ21}, 
$$\E\left[\int_0^t p^N_{u^-}d\A_u\right]=\E\left[\gamma_t^N(Q^2)\right]\leq\|\varphi\|_\infty^2.$$
So the right hand side of (\ref{lazeich}) is the product of an $\smallo_\P(1)$ 
with an $\bigo_\P(1)$, which is classically an $\smallo_\P(1)$ (see for example \cite{schervish}, Theorem 7.15, for a general version of this result), and the first term of (\ref{liashjec}) goes to zero in probability.\medskip

For the second term in (\ref{liashjec}), just notice that $\abs{p^N_{u^-}-p_u}p^N_{u^-}\leq 1$, so that It\^o isometry and (\ref{lemgQ24}) yields
\[\E\b{\p{\int_0^t (p^N_{u^-}-p_u) p^N_{u^-} d\tM_u}^2}=\E\b{\int_0^t(p^N_{u^-}-p_u)^2 \p{p^N_{u^-}}^2 d[\tM,\tM]_u}
\leq 5 \|\ph\|_\infty^4
\]
and $a^N_t$  tends to zero in probability as well.
\end{proof}

\subsection{Another formulation of the asymptotic variance}
In order to retrieve the expression of Theorem \ref{gamma}, one can then simplify the variance at final time $T$ as follows.
\begin{Lem}Define
$$
\sigma_T^2(\ph) \eqdef \Var_{\eta_0}\p{Q^T(\ph)} + i_T(\ph),
$$
with $i_T(\ph)$ like in~\eqref{eq:var_1}, then 
\begin{align}\label{eq:var_2}
\sigma_T^2(\ph) = p^2_T \Var_{\eta_T}(\ph) - p_T^2\ln(p_T) \, \eta_T(\ph)^2 - 2\int_0^T \Var_{\eta_{t}}(Q^{T-t}(\ph)) p_t dp_t.
\end{align}
\end{Lem}

\begin{proof} Since $\gamma_T(Q^2)=\gamma_T(\ph^2)$,
\begin{align}\label{stdeb}
i_T
&=p_T\gamma_T(\ph^2)-\gamma_0(Q^2)+\gamma_T(\ph)^2 \ln p_T-2\int_0^T\gamma_t\p{Q^2} dp_t.
 \end{align}
Furthermore, by definition,
\begin{align}\label{ptmoins1}
p_t^{-1}\gamma_t\p{Q^2}
=\eta_t(Q^{T-t}(\ph)^2)
=\Var_{\eta_{t}}(Q^{T-t}(\ph))+p_t^{-2}\gamma_t(Q^{T-t}(\ph))^2.
\end{align}
Recall that $\gamma_t(Q^{T-t}(\ph))=\gamma_T(\ph)$, so that reporting the latter identity into (\ref{ptmoins1}),
and then (\ref{ptmoins1}) into (\ref{stdeb}) gives
\begin{align*}
i_T
&=p_T\gamma_T(\ph^2)-\gamma_0(Q^2)-\gamma_T(\ph)^2 \ln p_T-2\int_0^T\Var_{\eta_t}(Q^{T-t}(\ph))p_t dp_t.
 \end{align*}
In the same way, $p_T\gamma_T(\ph^2)=p_T^2\Var_{\eta_T}(\ph)+\gamma_T(\ph)^2$ and $\Var_{\eta_0}\p{Q^T(\ph)}= \gamma_0(Q^2)-\p{\gamma_T(\ph)}^2$, hence the result.
\end{proof}

\subsection{Martingale Central Limit Theorem}
The following result is an adaptation of Theorem $1.4$ page 339 in~\cite{ek86} to our specific context. The main difference is about the initial condition.

\begin{The}\label{aeichj}
On a filtered probability space, let $t \mapsto z_t^N$ denote a sequence of c\`adl\`ag local martingales indexed by $N \geq 1$. Assume moreover that
\begin{enumerate}[(i)]
\item $z_0^N \xrightarrow[N \to + \infty]{\calD} \mu_0$, where $\mu_0$ is a given probability on $\R$.
\item One has $\lim_{N \to +\infty} \E[\sup_{t \in [0,T]}\abs{\Delta z^N_t }^2] =  0$.
\item For each $N$, there exists an increasing \cadlag process $t \mapsto i^N_t$ such that 
$t \mapsto \p{z_t^N-z^N_0}^2 -i^N_t$ is a local martingale.
\item The process  $t \mapsto i^N_t$ satisfies $\lim_{N \to +\infty} \E\left[\sup_{t \in [0,T]} \Delta i^N_t \right] =  0$.
\item There exists a continuous and increasing deterministic function $t \mapsto i_t$ such that, for all $t\in[0,T]$,
 \begin{align*}
  i^N_t \xrightarrow[N \to + \infty]{\P} i_t.
 \end{align*}
 \end{enumerate}
Then $(z_t^N)_{t \in [0,T]}$ converges in law (under the Skorokhod topology) to $(Z_t)_{t \in [0,T]}$, where  $Z_0 \sim \mu_0$ and  $(Z_t-Z_0)_{t \in [0,T]}$ is a Gaussian process, independent of $Z_0$, with independent increments and variance function $i_t$.  
\end{The}

\begin{proof}
First, we notice that Theorem $1.4$ with condition~$(b)$ in~\cite{ek86} is exactly the present result in the special case where $z_0^N=0$ and $\mu_0=\delta_0$. See also Section~$5$, Chapter~$7$ of~\cite{js03} in which, again, the case of a general initial condition is left to the reader. \medskip

Second, fix $\psi \in C_b(\R)$, and consider the $\P$-absolutely continuous probability defined by
$$ \P_\psi = \frac{1}{\E \b{  e^{\psi(z^N_0)} }} e^{\psi(z^N_0)}\ \P.$$
For any $\psi$, we claim that under $\P_\psi$ with the same filtration, all the assumptions of the present theorem hold for $t \mapsto z^N_t-z^N_0$ instead of $t \mapsto z^N_t$.\medskip

Indeed, first remark that since $\psi$ is bounded and since the probability $\P$ conditional on the initial $\sigma$-field is not modified, martingale properties under $\P$ still hold under under $\P_\psi$. The processes $t \mapsto z^N_t-z^N_0$ are thus local martingales under $\P_\psi$, with the same localizing stopping times. Since $\psi$ is bounded, the upper bound on jumps~$(ii)$ is satisfied. In addition, the process $t \mapsto \p{z^N_t-z^N_0}^2-i^N_t$ is still a local martingale and~$(iii)$ holds true. Again, since $\psi$ is bounded, the upper bound on jumps~$(iv)$ is satisfied. Finally, since $\psi$ is bounded, convergence in probability is independent of $\psi$, so that~$(v)$ is verified. \medskip

As a consequence, under each $\P_\psi$ with bounded $\psi$, the process $t \mapsto z^N_t-z^N_0$ converges in law under the Skorokhod topology to $(M_t)_{t \in [0,T]}$, a Gaussian martingale with initial value $M_0=0$ and variance function $i_t$. \medskip

Finally, let $F$ be a continuous functional on the Skorokhod space of c\`adl\`ag  paths, and $\psi$ a continuous bounded test function. Using the previous reasoning and assumption~$(i)$, we have that
\begin{align*}
 \E \b{ e^{\psi(z^N_0)} F(z^N_t-z^N_0, \, t \geq 0)} &= \E_\psi\b{ F(z^N_t-z^N_0, \, t \geq 0)} \E \b{e^{\psi(z^N_0)} } \\
 & \xrightarrow[N \to +\infty]{} \E \b{ F(M_t , \, t \geq 0 )  }\mu_0(e^\psi).
\end{align*}
Since $F$ and $\psi$ are arbitrary, the latter limit corresponds to the weak convergence of $(z^N_t)_{t \geq 0}$ towards $Z_t \eqdef Z_0 + M_t$, where $Z_0 \sim \mu_0$ and $(M_t)_{t \geq 0}$ are independent. This is exactly the desired result.
\end{proof}

\begin{Rem}
In other words, the limit Gaussian process $(Z_t)_{t \in [0,T]}$ is solution of the stochastic differential equation
$$\left\{\begin{array}{ll}
Z_0&\sim \mu_0\\
dZ_t&=\sqrt{i_t}\ dWt
\end{array}\right.$$
where $(W_t)_{t \in [0,T]}$ is a standard Brownian motion.
\end{Rem}

\begin{Pro}\label{lazicj} Under Assumption \ref{ass:E}, for any bounded $\ph$ such that Assumption~\ref{ass:Dp} is satisfied, the sequence of martingale $(z_t^N)_{0\leq t\leq T}$ defined by
 $$z_t^N=\sqrt{N}\left(\gamma_t^N(Q^{T-t}(\ph))-\gamma_0(Q^{T-t}(\ph))\right)$$
converges in law towards a Gaussian process $(Z_t)_{t \in [0,T]}$ with independent increments, initial distribution ${\cal N}(0,\Var_{\eta_0}(Q^T(\ph)))$ and variance function $\sigma_t^2(\ph)=\Var_{\eta_0}(Q^T(\ph))+i_t(\ph)$, with $i_t(\ph)$ defined by~\eqref{eq:var_1}.
 \end{Pro}
 
\begin{proof}
We just have to check that the assumptions of Theorem \ref{aeichj} are satisfied in our framework. Before proceeding, let us remind that since $\ph$ belongs to $\calD$, it is necessarily bounded.
\begin{enumerate}[(i)]
\item Recall that $(X_0^1,\ldots,X_0^N)$ are i.i.d.~with law $\eta_0=\gamma_0$, so that clearly
$$z_0^N=\sqrt{N}\left(\gamma_0^N(Q^{T}(\ph))-\gamma_0(Q^{T}(\ph))\right)\xrightarrow[N\to+\infty]{\cal D}{\cal N}(0,\Var_{\eta_0}(Q^T(\ph))).$$
\item This is a simple consequence of Corollary \ref{mzoecjleijf}.
\item This is the purpose of Lemma~\ref{lem:i_mart}.
\item By Definition \ref{def:i}, we have
$$i^N_t=\int_0^t\p{p^N_{u^-}}^2 d\A_u-\int_0^t\Var_{ \eta^{N}_{u^-}}(Q)p^N_{u^-}dp^N_u+\frac1N  \int_0^t \big(p^N_{u^-}\big)^2 d \calR_u,$$
so that
$$\Delta i^N_t\leq\Delta\A_t+\|\ph\|_\infty^2|\Delta p^N_t|+\frac1N|\Delta\calR_t|.$$ 
It remains to see that $|\Delta p^N_t|\leq 1/N$ and to apply the bounds given in (\ref{ajump}) and (\ref{lemab2}) to deduce that
$$\E\left[\sup_{0\leq t\leq T}\Delta i^N_t\right]\leq\frac{2\|\ph\|_\infty^2}{N}+\frac{14\|\ph\|_\infty^2}{N^2}\xrightarrow[N\to+\infty]{}0.$$
 \item This last - and most important - point is exactly Proposition~\ref{almchi}.
 \end{enumerate}
 \end{proof}

Let us assume in the following discussion that Assumption \ref{ass:E} is satisfied. If we marginalize on the final time, we obtain that, for any bounded $\ph$ such that Assumption~\ref{ass:Dp} is 
 satisfied,
$$\sqrt{N}\left(\gamma_T^N(\ph)-\gamma_T(\ph)\right)\xrightarrow[N\to\infty]{\cal D}{\cal N}(0,\sigma_T^2(\ph)).$$
In fact, we can extend this result to any function $\ph$ in the $\norm{\cdot}_\infty$-closure $\overline{\cal D}$ of $\calD$, and thus establish Proposition \ref{techprop}, and in turn Theorem \ref{gamma}. 

\begin{Lem}\label{alizhcaeoichj} 
Under Assumptions~\ref{ass:Dp} and \ref{ass:E}, for any $\ph \in \overline{\cal D}$, we have
\begin{align*}
\E \b{ \p{ \gamma^N_T(\ph) - \gamma_T(\ph) }^2 } \leq \frac{18 \norm{\ph}_\infty^2}{N}.
\end{align*}
\end{Lem} 

\begin{proof} For any $\ph$ in $\overline{\cal D}$, consider a sequence $(\ph_n)$ in $\cal D$ converging to $\ph$ with respect to the supremum norm. In particular, $(\|\ph_n\|_\infty)$ goes to $\|\ph\|_\infty$. Since $|\gamma_T(f)|\leq\|f\|_\infty$ and $|\gamma_T^N(f)|\leq\|f\|_\infty$, we have
\begin{align*}
\p{ \gamma^N_T(\ph) - \gamma_T(\ph) }^2&\leq 3\left\{ \gamma^N_T(\ph-\ph_n) ^2+\p{ \gamma^N_T(\ph_n) - \gamma_T(\ph_n) }^2+ \gamma_T(\ph_n-\ph) ^2\right\}\\
&\leq 3\p{ \gamma^N_T(\ph_n) - \gamma_T(\ph_n) }^2+6\|\ph-\ph_n\|_\infty^2.
\end{align*}
Now, Proposition \ref{pro:estimate} implies
$$\E \b{ \p{ \gamma^N_T(\ph) - \gamma_T(\ph) }^2 } \leq \frac{18 \norm{\ph_n}_\infty^2}{N}+6\|\ph-\ph_n\|_\infty^2\xrightarrow[n\to+\infty]{}\frac{18 \norm{\ph}_\infty^2}{N}.$$
\end{proof}

Since \ref{ass:D} implies \ref{ass:Dp} by Lemma \ref{mzoecj}, the next result is exactly Theorem \ref{gamma}.

\begin{Cor}\label{corbis}
Under Assumptions~\ref{ass:Dp} and \ref{ass:E}, for any $\ph \in \overline{\cal D}$, one has 
$$\sqrt{N}\left(\gamma_T^N(\ph)-\gamma_T(\ph)\right)\xrightarrow[N\to\infty]{\cal D}{\cal N}(0,\sigma_T^2(\ph)).$$
 \end{Cor}

\begin{proof}
We will use  the simplified version~\eqref{eq:var_2} of the asymptotic variance. 
Let us denote by $\Phi$ any bounded Lipschitz function, $G$ a centered Gaussian variable with variance $\sigma_T^2(\varphi)$ for an arbitrary  function $\varphi\in\overline{\cal D}$. \medskip

For any  $\varepsilon>0$, we can find $\varphi_\varepsilon$ in $ \calD$ such that $\|\varphi-\varphi_\varepsilon\|_\infty\leq\varepsilon$. We can also assume that $\gamma_T(\varphi_\varepsilon)=\gamma_T(\varphi)$. Note that we can also choose $\varphi_\varepsilon$ such that $|\sigma_T^2(\varphi_\varepsilon)-\sigma_T^2(\varphi)|\leq \varepsilon$. Indeed, it is easy to check by dominated convergence that $\varphi\mapsto\sigma_T^2(\varphi)$ is continuous for the norm $\|\cdot \|_\infty$. Hence, let us denote by $G_\varepsilon$ a centered Gaussian variable with variance $\sigma_T^2(\varphi_\varepsilon)$.\medskip

Then we may write
\begin{align*}
|\E[\Phi(\sqrt{N}&(\gamma_T^N(\varphi)-\gamma_T(\varphi))]-\E[\Phi(G)]|\\
\leq&\ \E[|\Phi(\sqrt{N} (\gamma_T^N(\varphi)-\gamma_T(\varphi)))-\Phi(\sqrt{N} ( \gamma_T^N(\varphi_\varepsilon)-\gamma_T(\varphi)))|]\\
&+|\E[\Phi(\sqrt{N} ( \gamma_T^N(\varphi_\varepsilon)-\gamma_T(\varphi)))]-\E[\Phi(G_\varepsilon)]|\\
&+|\E[\Phi(G_\varepsilon)]-\E[\Phi(G)]|.
\end{align*}

For the first term, by Lemma \ref{alizhcaeoichj}, Jensen's inequality and remembering that $\gamma_T(\ph-\varphi_\varepsilon)=0$, we have
$$\E|\Phi(\sqrt{N} (\gamma_T^N(\varphi)-\gamma_T(\varphi)))-\Phi(\sqrt{N} ( \gamma_T^N(\varphi_\varepsilon)-\gamma_T(\varphi)))|\leq3\sqrt{2}\|\Phi\|_{\rm Lip}\|\varphi-\varphi_\varepsilon\|_\infty.$$
Hence, for any given $\delta>0$, we can choose $\varepsilon$ such that this first term is less than $\delta$. Clearly, the same property holds for the third term as well. Besides, since $\ph_\varepsilon$ is in $\calD$, for $N$ large enough, the second term can also be made less than $\delta$ by Corollary \ref{lazicj}. As this result holds for any bounded Lipschitz function $\Phi$, we conclude using the Portmanteau theorem.

\end{proof}

 \begin{Rem}
  This corollary is particularly useful in practice: to obtain the CLT associated with any observable $\ph$, it is sufficient to check Assumption~\ref{ass:D} or~\ref{ass:Dp} for appropriately regularized functions. 
 \end{Rem}

\section{Appendix}

\subsection{Preliminary on Feller processes}\label{proufefellerd}
In this section, we recall the definition and some properties of Feller processes (see also for example Section~$17$ of~\cite{Kall}).

\begin{Def} Let $E$ be a locally compact Polish space. Let $C_0(E)$ denote the space of continuous functions that vanish at infinity. A \cadlag time-homogeneous  process in $E$ is Feller if and only if each of its probability transitions maps $C_0(E)$ into itself. Formally: for all $\ph\in C_0(E)$ and $t \geq 0$, $z \mapsto \E[\ph(Z_t) | Z_0=z] \in C_0(E)$, where $(Z_t)_{t \geq 0}$ denotes the Markov process constructed with any given initial condition $Z_0=z \in E$.
\end{Def}

Feller processes enjoy many useful standard properties including: (i) The associated natural filtration $\calF^Z_t \eqdef \sigma\p{Z_{t'}, \, 0 \leq t' \leq t}$ is right-continuous; (ii) $Z$ is strong Markov with respect to $\calF^Z$; %
(iii) $Z$ is quasi-left continuous with respect to $\calF^Z$. A characterization of quasi-left continuity is the following (\cite{js03}, Proposition~$2.26$): if $(\tau_n)_{n \geq 1}$ is any increasing sequence of stopping times, then on the event $\set{ \lim_n \tau_n < + \infty }$, one has $\lim_n Z_{\tau_n} =  Z_{\lim_n \tau_n}$. Note that taking deterministic sequences implies that quasi-left continuous processes never jump at deterministic times. \medskip

We will need a slightly less standard property of Feller processes related to the so-called Skorokhod $J_1$ topology as defined in the following proposition.

\begin{Pro}[$J_1$ topology]
Let $d$ be a metric of the Polish topology of $E$. Let $\D_E$ denote the space of \cadlag maps from $\R_+$ to $E$. There is a Polish topology on $\D_E$, called the Skorokhod $J_1$ topology, characterized by the following property: $\lim_n (z^n_t)_{t \geq 0} = (z_t)_{t \geq 0}$ in $\D_E$ if and only if there is a sequence $(\lambda^n)_{n \geq 0}$ of increasing one-to-one maps of $\R_+$ onto itself such that for each $t_0 \geq 0$
$$\lim_n \sup_{0 \leq t \leq t_0} d \p{z^n_{\lambda^n(t)},z_t} =\lim_n \sup_{0 \leq t \leq t_0} \abs{\lambda^n(t)-t} =0.$$
\end{Pro}

If $Z$ is Feller, the distribution of $(Z_t)_{t \geq 0} \in (\D_E,J_1)$ is continuous with respect to its initial condition $Z_0=z$. This is detailed in the following lemma.

\begin{Lem}\label{lem:cont_feller}
Let $\D_E$ denote the space of \cadlag trajectories endowed with the Skorokhod $J_1$ topology, and let $\p{Z^{z}_{s}}_{s \geq 0}$ denote a Feller process with initial condition $Z_0=z$. The map $ z \mapsto {\cal L}{\p{Z^{z}_{t}}_{t \geq 0}}$ defined from $E$ to probabilities on $\D_E$, endowed with convergence in distribution, is continuous.
\end{Lem}
\begin{proof}
 Let $(z^n)_{n \geq 0}$ be a sequence of initial conditions with $\lim_n z^n = z$ and denote $Z^n \eqdef Z^{z^n}$ as well as $Z \eqdef Z^{z}$. Then, by Condition~(iv) of Theorem~$17.25$ in~\cite{Kall} (Condition~(ii), which implies Condition~(iv), is trivially true in the present context), the sequence of processes $\p{ \p{Z^{n}_{t}}_{t \geq 0}}_{n \geq 0}$ converges in distribution towards $\p{Z_{t}}_{t \geq 0}$ in the Skorokhod space $(\D_E,J_1)$. 
\end{proof}

We then recall the lower and upper continuity of hitting times with respect to the Skorokhod $J_1$ topology.

\begin{Lem}\label{lem:stop_time_cont} Let $B\subset E$, $(z_t)_{t \geq 0}\in \D_E$, and define $t_B(z) \eqdef \inf \{t \geq 0,\   z_t \in \mathring{B}\}$,
 as well as $\overline{t}_B(z)  \eqdef \inf \{t \geq 0,\  z_{t^-} \in \overline{B} \text{\, or \,} z_t \in \overline{B}\}$.
 Consider a converging sequence $\lim_n (z^n_t)_{t \geq 0} = (z_t)_{t \geq 0}$ in $(\D_E,J_1)$. Then $t_B$ is upper continuous in $(\D_E,J_1)$:
 $$ \limsup_n t_B \p{z^n} \leq t_B\p{z},$$
 and $\overline{t}_B$ is lower continuous in $(\D_E,J_1)$:
 $$\overline{t}_B \p{z} \leq \liminf_n \overline{t}_B \p{z^n}.$$ 
\end{Lem}
\begin{proof}
For the upper continuity, without loss of generality, we can assume that $t_B\p{z} < +\infty$. By right continuity of $(z_t)_{t \geq 0}$ and by definition of $t_B$, $z_{ t_B\p{z}+\eps} \in \mathring{B}$ for some arbitrary small enough $\eps > 0$. By definition of the Skorokhod topology, there is a converging sequence $\lim_n {t_n} = t_B\p{z}+\eps$ in $\R^+$ such that $\lim_n z^n_{t_n} = z_{t_B\p{z}+\eps} \in \mathring{B}$. Thus, since $\mathring{B}$ is open, for any $n$ large enough, $z^n_{t_n} \in \mathring{B}$ so that $t_B\p{z^n} \leq t_n$. The result follows by taking the limit $n \to + \infty$ and then $\eps \to 0$, $\eps$ being arbitrary. \medskip
 
Concerning the lower continuity, set $t_0 \eqdef \liminf_n \overline{t}_B \p{z^n}$, which we assume finite without loss of generality. By definition of the hitting time functional $\overline{t}_B$, we can construct a sequence $(t_n)_{n \geq 1}$ such that, up to extraction, (i) $t_n \leq t_0 +1$, (ii) $\lim_n t_n =t_0$, and (iii) $\lim_n d(z^n_{t^n},B) =0$ where $d$ denotes a distance for the Polish space $E$. On the other hand, by time uniformity in the definition of the $J_1$ convergence $(z^n)_{t \geq 0} \to (z_t)_{t \geq 0}$, the set $\set{z^n_t,\ t \leq t_0+1, \, n \geq 0}$ is bounded. Hence, by compacity, there exists a sub-sequence of $(t_n)_{n \geq 1}$ satisfying $ \lim_n t_n = t_0$ as well as $z^n_{t_n}\rightarrow b$, where $b \in \overline B$ by condition~(iii) above. The convergence in $J_1$ topology implies that the extracted limit $b$ necessarily belongs to $\{z_{t_0^-},z_{t_0}\}$, implying that either $z_{t_0^-} \in \overline B$ or $z_{t_0} \in \overline B$. By definition of $\overline{t}_B$, this means that  $\overline{t}_B \p{z} \leq t_0$.
\end{proof}

We can conclude with the key property that is useful in the proof of Proposition~\ref{pro:Feller_D}. %

\begin{Lem}\label{lem:time_cont_final}
Let $B$ be a subset of $E$, $Z$ a Feller process, and $z \in E$ a given initial condition. Denote $
\overline{\tau}_B  \eqdef \inf \{t \geq 0,\  Z_{t^-} \in \overline{B} \text{\, or \,} Z_t \in \overline{B}\} \in [0,+\infty]$ as well as $\tau_B \eqdef \inf \{t \geq 0,\   Z_t \in \mathring{B}\}\in [0,+\infty]$. Besides, assume that
  \begin{equation}\label{eq:time_cond}
    \P \p{\overline{\tau}_B = \tau_B | Z_0=z} =1.
  \end{equation}
Let $\lim_n z^n= z $ be a given converging sequence of initial conditions. Then the distribution of $\tau_B \in [0,+\infty]$ under $\P\p{ \,\, . \,\, | Z_0=z^n}$ converges when $n \to +\infty$ 
  towards its distribution under $\P\p{\,\, . \,\, | Z_0=z}$. Moreover if $\P( \tau_B < + \infty| Z_0=z ) > 0$, then the distribution of $(Z_{\tau_B},\tau_B)$ under $\P\p{ \,\, . \,\, | Z_0=z^n, \, \tau_B < +\infty}$ converges when $n \to +\infty$ 
  towards its distribution under $\P\p{\,\, . \,\, | Z_0=z, \, \tau_B < +\infty}$. 
\end{Lem}

\begin{proof}

Using Lemma~\ref{lem:cont_feller} and a Skorokhod embedding argument, we can construct a sequence $(Z^n_t)_{t \geq 0}$ of processes with initial conditions $(z^n)_{n \geq 0}$ such that $\lim_n Z^n = Z$ in $(\D_E,J_1)$ almost surely. We claim that (i) $\lim_n \tau^n_B = \tau_B$, and (ii) $\lim_n Z^n_{\tau^n_B} = Z_{\tau_B}$ on the event $\set{\tau_B < +\infty}$, which enable to conclude by dominated convergence. \medskip 

On the one hand, Lemma~\ref{lem:stop_time_cont} together with~\eqref{eq:time_cond} directly implies (i). \medskip 

On the other hand, let us work on the event $\set{\tau_B < +\infty}$. The definition of the Skorokhod topology implies that the sequence $(Z^n_{\tau^n_B})_{n \geq 1}$ has its accumulation points included in $\{ Z_{\tau^\infty_B-} , Z_{\tau^\infty_B}\}$. Since by construction $Z^n_{\tau^n_B} \in \overline B$, these accumulation points are also included in $\overline B$. 
We now claim that by quasi-left continuity of $Z$ and condition~\eqref{eq:time_cond}, $Z_{\tau_B^-} \in \overline B \Rightarrow Z_{\tau_B^-}=Z_{\tau_B}$, which in turn implies from the discussion above that $\lim_n Z^n_{\tau^n_B} = Z_{\tau_B}$, and hence proves (ii) above. Indeed, defining $B_{k} = \{x,\ d(x,B) < 1/k\}$, one has by construction $\lim_k \tau_{B_k}=\overline{\tau}_B$ which also equals $\overline{\tau}_B = \tau_B$ by Assumption~\eqref{eq:time_cond}. Then quasi-left continuity implies that $\lim_k Z_{\tau_{B_k}}= Z_{\lim_k \tau_{B_k}}=Z_{\tau_{B}}$, while $Z_{\tau_B^-} \in \overline B$ implies $\tau_{B_k} < \tau_B$ so that $\lim_k Z_{\tau_{B_k}} = Z_{\tau_B^-}$, hence the claimed result.%
\end{proof}

\begin{proof}[Proof of Proposition~\ref{pro:Feller_D}]

The Feller property classically implies the quasi-left continuity of $t \mapsto \widetilde{X}$ and thus Condition~(i) of Assumption~\ref{ass:D} for all jump times except perhaps $\tau_\partial$. \medskip

Let $(x^n,t^n)$ be a sequence in $F \times [0,T]$ converging to $(x,t) \in F \times [0,T]$. 
We claim that 
\begin{equation}\label{eq:cont} 
\lim_n \E_{x^n}\b{ \ph(\widetilde{X}^n_{t_n}) \un_{\tau_\partial^n> t^n}} 
= \E_{x}\b{ \ph(\widetilde{X}_{t}) \un_{\tau_\partial>t}},
\end{equation}
which will ensure Condition~(ii) of Assumption~\ref{ass:D}.\medskip

First, we claim that $\P_{x}(\tau_\partial=t)=0$. 
Indeed, since $\widetilde{X}$ is Feller hence quasi left continuous, it cannot jump at a given $t \geq 0$ so that
$\{\tau_\partial=t\}= \{\tau_\partial=t\ \text{and}\ \widetilde{X}_t= \widetilde{X}_{t^-}\}$. Thus $\set{\tau_\partial=t}$ implies $\widetilde{X}_{t} \in \partial F$, which has probability zero by Condition~(i) in Proposition~\ref{pro:Feller_D}.
\medskip 

Second, we claim that $$
\P_{x}\p{ \overline{\tau}_\partial =\tau_\partial} = 1.
$$
where $\overline{\tau}_\partial \eqdef \inf \{t,\ \widetilde{X}_{t^-} \in E \setminus F \, \text{or} \, \widetilde{X}_{t} \in E \setminus F \} < \tau_\partial$. Indeed, by the strong Markov property of Feller processes, it is enough to prove that $\P_{\widetilde{X}_{\overline{\tau}_\partial}}(\tau_\partial>0)=0$, which is just a consequence of Condition~(ii) in Proposition~\ref{pro:Feller_D}. \medskip

Finally, according to Lemma~\ref{lem:cont_feller}, a Skorokhod embedding argument shows that we can assume the almost sure convergence $\lim_n \widetilde{X}^n = \widetilde{X}$ in $(\D_E,J_1)$. Since $\widetilde{X}$ is Feller hence quasi left continuous, $\lim_n \widetilde{X}^n_{t_n} = \widetilde{X}_{t}$. To obtain~\eqref{eq:cont}, it remains to show that $\lim_n \tau_\partial^n = \tau_\partial$. This follows from Lemma~\ref{lem:time_cont_final} by simply taking $B=E \setminus \overline F$.

\end{proof}

\subsection{Stopping times and martingales}
\begin{Lem}\label{albcios}
 Let $\tau$ be a stopping time on a filtered probability space, and $U$ an integrable and $\calF_\tau$ measurable random variable such that $\E \b {U | \calF_{\tau^-}}=0$. Then the process $t \mapsto U \un_{t \geq  \tau}$ is a c\`adl\`ag martingale. 
\end{Lem}
\begin{proof}
 Let $t > s$ be given. First remark that $\un_{t \geq  \tau} = \un_{s \geq  \tau} +  \un_{s <  \tau} \un_{t \geq  \tau} $. Then by definition of $\calF_\tau$, $U\un_{s \geq  \tau} $ is $\calF_s$-measurable, so that
 \begin{align*}
 \E \b{ U \un_{t \geq  \tau} | \calF_s} =  U\un_{s \geq  \tau} + \E \b{U \un_{t \geq  \tau} 
  | \calF_s}\un_{s <  \tau}.
  \end{align*}
Next, by definition of $\calF_{\tau^-}$, $\E \b{U \un_{t \geq  \tau} 
  | \calF_s}\un_{s <  \tau} $ and $\un_{t \geq  \tau}$ are $\calF_{\tau^-}$-measurable, so that
  \begin{align*}
   \E \b{U \un_{t \geq  \tau} 
  | \calF_s}\un_{s <  \tau}  =\E \b{\E \b{ U | \calF_{\tau^-} }\un_{t \geq  \tau} 
  | \calF_s}\un_{s <  \tau}  =0.
   \end{align*}
The result follows.
\end{proof}

\subsection{Proof of Lemma~\ref{mzoecj}: \ref{ass:D} $\Rightarrow$ \ref{ass:Dp}}\label{lackalicj}

The following obvious weakening of Assumption~\ref{ass:D} is the raw condition that is required in the proof of Lemma~\ref{mzoecj}.

\begin{itemize}
\item[(1)] For any initial condition $ x \in F$, the killing time has an atomless distribution, that is
$$\P\p{ \tau_\partial = t|X_0=x} = 0 \qquad \forall t \geq 0.$$
\item[(2)] There exists a space $\calD$ of bounded measurable real-valued functions on $F$,
which contains at least the indicator function $\un_F$, and such that for any $\ph \in \calD$, for any initial condition $x \in F $, the jumps of the \cadlag version of the martingale $t \mapsto Q^{t_0-t}(\ph)(X_{t})$ have an atomless distribution:
$$
\P\p{ \Delta Q^{t_0-t}(\ph)(X_{t}) \neq 0 | X_0 =x} = 0 \qquad  \forall  0 \leq t \leq t_0.
$$
\end{itemize}

Our goal now is to prove that conditions (1) and (2) above imply Assumption \ref{ass:Dp}. Throughout the proof, let $1\leq m \neq n \leq N$ and $j,k \geq 0$ be given integers. We recall that, by convention, $\tau_{n,0}=\tau_{m,0}=0$.\medskip

(i) It is sufficient to prove that $\P\p{\tau_{n,k+1} = \tau_{m,j+1} \ \& \  \tau_{m,j} \leq \tau_{n,k} }=0$, since taking the countable union of such events over $j,k \geq 0$ and $1\leq m \neq n \leq N$ will yield the result. Conditionally on $\calF_{\tau_{n,k} }$ and $\set{\tau_{m,j} \leq \tau_{n,k} }$, the two branching times $\tau_{n,k+1}$ and $\tau_{m,j+1}$ are independent. Moreover, Assumption (1) implies that conditionally on $\calF_{\tau_{n,k}}$, $\tau_{n,k+1}$ has an atomless distribution. We deduce that 
$$\P(\tau_{n,k+1} = \tau_{m,j+1} \ \& \  \tau_{m,j} \leq \tau_{n,k} | \calF_{\tau_{n,k} })=0.$$ 

(ii) According to Proposition 1.3 in \cite{js03}, we can define a countable sequence of stopping times $\sigma_{m,a}$ with $ a\geq 1$ that exhaust the jumps of $\L^m_t$ for $ \tau_{m,j} \vee \tau_{n,k} \leq t \leq \tau_{m,j+1} $.  Conditionally on $\calF_{\tau_{n,k}}$ and $\set{\tau_{m,j} \leq \tau_{n,k}}$, the two processes $(\L^n_t)_{t < \tau_{n,k+1}}$ and $(\L^m_t)_{t < \tau_{m,j+1}}$ are independent. Moreover, Assumption~(2) implies that conditionally on $\calF_{\tau_{n,k}}$, $\p{\L^n_{t} = Q^{T-t}(\ph)(X^n_t)}_{\tau_{n,k} \leq t < \tau_{n,k+1}}$ has jumps with atomless distribution. %
As a consequence, for each $a\geq 1$,  
$$\P\p{\left.\Delta \L^n_{\sigma_{m,a}} \neq 0 \ \& \ \tau_{m,j} \leq \tau_{n,k} \right|\calF_{\tau_{n,k}} }= 0.$$ Taking the countable union of such events over $a\geq 1$, $j,k \geq 0$ and $1\leq m \neq n \leq N$ gives the result.\medskip

(iii) One can apply the same reasoning as for (ii) with $\tau_{m,j+1}$ instead of $\sigma_{m,a}$.

\subsection{Integration rules}

Remember that $p^N_t=(1-1/N)^{N\calN_t}$, so that $p^N_0=1$. Recall that $\sum_{j=1}^\infty (1-1/N)^{j-1}=N$.%

\begin{Lem}\label{oachz} Assume $N \geq 2$.
Let $t \mapsto z^N_t$ be a \cadlag semi-martingale, $c>0$ a deterministic constant, and consider the following conditions, satisfied for any branching time $\tau_j$, $j \geq 1$:
$$(i)\  |\Delta z_{\tau_j}^N| \leq c/N,\ (ii)\ |z^N_{\tau_{j}^-}| \leq c(1-1/N)^j,\ (iii)\ |\Delta z_{\tau_j}^N| \leq c(1-1/N)^j/N.$$
If~(i) holds true, one has
\begin{align}
\int_0^t   p^N_{s^-} dz_{s}^N     =  p^N_tz_t^N - z_0^N -  \int_0^t z_{s^-}^N dp^N_s+ O\p{1/N}. \label{eq:ipp1}
\end{align}
If (ii)~holds true, one has
\begin{equation}\label{eq:ipp2}
 \int_0^t z^N_{s^-} \p{p^N_{s^-}}^{-1} dp^N_s = \int_0^t z^N_{s^-} d\ln p^N_s + O\p{1/N}. 
\end{equation}
Finally, if~(iii) holds true, one has
\begin{equation}\label{eq:ipp3}
  \int_0^t z^N_{s^-} d\ln p^N_s  = z^N_t \ln p^N_t -  \int_0^t \ln p^N_{s^-} dz_s^N + O \p{1/N}.
\end{equation}
In all equations above, the $O$ notation only depends on the deterministic constant $c$.
\end{Lem}
\begin{proof}
 Equation~\eqref{eq:ipp1} comes from the integration by parts formula defining the quadratic variation %
 $$ p^N_tz_t^N - p^N_0z_0^N=\int_0^t z_{s^-}^N dp^N_s+\int_0^t   p^N_{s^-} dz_{s}^N+[p^N,z^N]_t,$$
 and the fact that $\Delta p_{\tau_j}^N=-(1-1/N)^{j-1}/N$ for all $j\geq 1$ so that
 $$[p^N,z^N]_t=\sum_{j \geq 1}\Delta p_{\tau_j}^N\Delta z_{\tau_j}^N=O\p{1/N}.$$ 
For \eqref{eq:ipp2}, notice that for any jump time $\tau_j$, $j \geq 1$, one has  $\p{p^N_{\tau_j^-}}^{-1} \Delta p^N_{\tau_j}=-\frac1N$
 as well as $ \Delta \ln p^N_{\tau_j}=\log(1-\frac1N)$, implying that
 \begin{equation*}
 \abs{ \int_0^t z^N_{s^-} \b{\p{p^N_{s^-}}^{-1} d p^N_s- d\ln p^N_s} } \leq \sum_j |z^N_{\tau_j^-}| \abs{\log(1-\frac1N)+\frac1N}  = O\p{1/N}. 
\end{equation*}
Similarly to \eqref{eq:ipp1}, Equation~\eqref{eq:ipp3} is merely an integration by parts formula, with this time
\begin{align*}
&[\ln p^N,z^N]_t=\sum_{j\geq 1}\Delta \ln p_{\tau_j}^N\Delta z_{\tau_j}^N= \log(1-\tfrac{1}{N})\sum_{j}\Delta z_{\tau_j}^N=O\p{1/N}.\qedhere
\end{align*}
\end{proof}

\bibliographystyle{plain}
\bibliography{biblio-cdgr}
\end{document}